\documentclass[smallextended]{svjour3}
\usepackage{amsmath, amssymb}
\usepackage{newtxtext,newtxmath}

\usepackage{enumerate, color}
\usepackage{caption}
\usepackage{subcaption}
\usepackage{graphicx}
\usepackage{ulem} 
\usepackage{multirow}
\usepackage{makecell}
\usepackage{url}     
\usepackage{framed}  
\usepackage[figuresright]{rotating} 
\usepackage{float}   
\usepackage{longtable}
\usepackage{tabularx}
\usepackage{array}
\usepackage{natbib}  
\usepackage[colorlinks=true,citecolor=blue]{hyperref}
\usepackage[misc]{ifsym} 

\smartqed

\newtheorem{assumption}{Assumption}
\numberwithin{equation}{section}

\title{Square-Root Clustered Lasso for High-Dimensional Linear Regression with Finite-Sample Theory and Efficient Algorithm Implementation}
\titlerunning{Square-Root Clustered Lasso}

\author{Shuqi Wang \and Yunhai Xiao\,{\normalfont\Letter}\thanks{Corresponding author, ORCID: \href{https://orcid.org/0000-0002-7503-4585}{0000-0002-7503-4585}} \and Mengjiao Shi}

\institute{Shuqi Wang \and Mengjiao Shi \at
	School of Mathematics and Statistics, Henan University, Kaifeng 475000, P.R. China\\
	\email{sqwang@henu.edu.cn},\,\email{smj@henu.edu.cn}
	\and
	Yunhai Xiao \at
	Center for Applied Mathematics of Henan University, Henan University, Zhengzhou 450046, P.R. China\\
	\email{yhxiao@henu.edu.cn}
}

\begin{document}
	\maketitle
	\begin{abstract}
	This paper proposes a square-root clustered Lasso penalized method for high-dimensional linear regression. By integrating square-root loss with the clustered Lasso penalty, the resulting method enables flexible regression modeling and allows tuning parameter selection independent of the unknown noise level. We establish weak decomposability of the clustered Lasso penalty and a noise-level-free tuning strategy under Gaussian noise. Furthermore, we establish the finite-sample prediction accuracy guarantees for the proposed method.
	The inherent nonsmoothness of the square-root loss and clustered Lasso penalty poses considerable computational difficulties. To resolve this issue, we develop an efficient semismooth Newton-based proximal augmented Lagrangian algorithm (N-ALM) to solve the corresponding dual optimization problem. The subproblems arising in the iteration are efficiently tackled by the semismooth Newton method, which fully utilizes the sparsity structure of the generalized Jacobian. We further establish global convergence and a local linear rate for the distance to the KKT solution set under suitable conditions; the local rate becomes superlinear when the penalty parameter tends to infinity.
	We conduct numerical experiments on real-world datasets to evaluate the computational performance of N-ALM. The results demonstrate that N-ALM solves the tested high-dimensional SRCL problems accurately and substantially reduces running times compared with d-ADMM.
	\end{abstract}
	
	
	\section{Introduction}\label{intro}
	\setcounter{equation}{0}
    In contemporary statistics and machine learning studies, variable selection for linear regression in high dimensional feature space has played a pivotal role.
    Let $(x_1, y_1), \dots, (x_n, y_n)$ be data generated from a linear regression model
    \begin{equation}\label{lineareg}
        y_i = x_i  \beta^0 + \sigma \epsilon_i,
    \end{equation}
    where $x_i^\top = (x_{i1}, \dots, x_{ip})^{\top} \in \mathbb{R}^p$ is the $i$th predictor vector, $y_i \in \mathbb{R}$ is the corresponding response variable, and $\epsilon_i$'s are independent and identically distributed (i.i.d.) noises such that $\mathrm{E}(\epsilon_i) = 0$ and $\mathrm{E}(\epsilon_i^2) = 1$. The vector $\beta^0 \in \mathbb{R}^p$ is the unknown true regression coefficient, and $\sigma > 0$ is the unknown noise level. 
    We let $X := (x_1^\top, \dots, x_n^\top)^{\top} \in \mathbb{R}^{n \times p}$ be the matrix of features and $y := (y_1, \dots, y_n)^{\top} \in \mathbb{R}^n$ be the response vector. In high dimensional regression problems, the dimension of predictors $p$ is usually much larger than the sample size $n$. With such a large number of predictors, one often prefers to select a smaller subset that is significant and highly correlated to the response for the ease of interpreting the model, as well as to avoid the issue of overfitting.

    One of the most popular methods for variable selection is the Lasso penalized least-squares \citep{tibshirani1996regression}:
    \begin{equation} \label{classic_lasso_model} 
        \min_{\beta \in \mathbb{R}^p} \left\{ \frac{1}{n} \|y - X\beta\|^2_2 + \lambda \|\beta\|_1 \right\},
    \end{equation}
    where $y \in \mathbb{R}^n$ is a given vector, and $\lambda > 0$ is the tuning parameter, and $\|\cdot\|_2$ and $\|\cdot\|_1$ denote the $\ell_2$-norm and $\ell_1$-norm, respectively. This method performs simultaneous variable selection and parameter shrinkage in order to enhance the prediction accuracy and interpretability of the resulting statistical model by minimizing the residual sum of squared errors plus an $\ell_1$-norm penalty term. 
    Despite its popularity, the classic Lasso model has two notable drawbacks. First, while it promotes sparsity in the regression coefficients, it does not exploit possible structural patterns among the nonzero coefficients. Second, the tuning parameter $\lambda$ depends on the unknown noise level $\sigma$, making its selection challenging in practice. Indeed, the theoretically optimal tuning parameter is of order $\sigma\sqrt{(\log p)/n}$.
    In classical low-dimensional linear regression, the noise level $\sigma$ is often estimated from the residual sum of squares, namely,
   $\hat{\sigma}^2={\|y-X\hat{\beta}\|_2^2}/{(n-p)}$,
   where $\hat{\beta}$ denotes the ordinary least-squares estimator. However, this approach is no longer applicable in high-dimensional settings with $p>n$, making accurate estimation of $\sigma$ considerably more challenging.
   
    To address the first limitation, \citet{She2010Sparse} introduced a sparse regression with exact clustering that achieves both sparsity and exact coefficient clustering.   It should be noted that this method has been applied to microarray data analysis and can also be used as a preprocessing step to uncover latent grouping structures among predictors before applying methods such as the fused Lasso \citep{tibshirani2005fused} and the group Lasso \citep{yuan2006model}.    
    Moreover, various algorithms have been developed for solving the method of \citet{She2010Sparse} , ranging from early approaches based on constrained reformulations with auxiliary variables \citep{Petry2011Pairwise,tang2016fused} to the efficient semismooth Newton augmented Lagrangian method (SSNAL)  \citep{lin2019efficient}.

    To address the second limitation, \cite{belloni2011square} introduced the square-root Lasso penalized method
    \begin{equation}\label{square_root_lasso_model}
		\min_{\beta \in \mathbb{R}^p} \Big\{ \frac{1}{\sqrt{n}}\|y - X\beta\|_2 + \lambda \|\beta\|_1 \Big\}.
	\end{equation}
    In contrast to the classic Lasso model \eqref{classic_lasso_model}, the choice of $\lambda$ in model \eqref{square_root_lasso_model} is independent of the unknown noise level $\sigma$. 
    More generally, square-root Lasso-penalized method takes the following form
    \begin{equation}\label{square_root_loss_problem}
		\min_{\beta \in \mathbb{R}^p} \Big\{ \frac{1}{\sqrt{n}}\|y - X\beta\|_2 + \lambda p(\beta) \Big\},
	\end{equation}
    where $p(\cdot)$ is the penalty term. 
    A key theoretical result established by \cite{stucky2017sharp} shows that if $p(\cdot)$ is a weakly decomposable norm \citep[Definition 4.1]{van2014weakly}, then the optimal tuning parameter $\lambda$ in \eqref{square_root_loss_problem} is independent of the unknown noise level $\sigma$, and the resulting estimator enjoys sharp oracle properties. 
    In fact, \cite{stucky2017sharp} showed that this property holds for several estimators, including the square-root Lasso \citep{belloni2011square}, the square-root group Lasso \citep{belloni2011square,yuan2006model}, and the square-root SLOPE \citep{bogdan2015slope}.
    Nevertheless, it is worth noting that \cite{jiang2021simultaneous} for square-root fused Lasso and \cite{li2026graph} for square-root graph-based penalty separately established this property that tuning parameters can be selected without estimating the noise level, without requiring the weakly decomposable norm condition.

    To achieve simultaneous variable selection and coefficient clustering while retaining the scale-free property of the square-root loss, in this paper, we propose the following square-root clustered Lasso (SRCL) method:
    \begin{equation}\label{square_root_clustered_lasso_model}
		\min_{\beta \in \mathbb{R}^p} \left\{ \frac{1}{\sqrt{n}}\|y - X\beta\|_2 + \lambda_1\|\beta\|_1 + \lambda_2 \sum_{1 \le i < j \le p} |\beta_i - \beta_j| \right\},
	\end{equation}
	where $\lambda_1, \lambda_2 > 0$ are given tuning parameters. 
    This model can be viewed as a square-root extension of the sparse regression model with exact clustering of \citet{She2010Sparse}, where the quadratic loss is replaced by the square-root loss of \citet{belloni2011square}.  
    Specifically, compared to the square-root Lasso method \eqref{square_root_lasso_model}, model \eqref{square_root_clustered_lasso_model} incorporates an additional pairwise difference penalty term $\sum_{i < j } |\beta_i-\beta_j|$, which encourages coefficients with similar effects to be grouped together.   
    In what follows, we prove that the penalty $p(\beta)=\lambda_1\|\beta\|_1+\lambda_2\sum_{i<j}|\beta_i-\beta_j|$ is weakly decomposable with respect to an index set $S$ whenever $\lambda_1>\lambda_2|S|$. This property permits a penalty-adapted analysis of the prediction error.    
    We also show that the square-root loss function allows the tuning parameters $\lambda_1$ and $\lambda_2$ to be selected independently of the unknown noise level $\sigma$ in \eqref{lineareg}.
    Moreover, we establish finite-sample guarantees for prediction error to provide a theoretical characterization of the statistical performance of the proposed estimator.
    Despite its attractive statistical properties, the SRCL model is substantially more challenging to solve. The main difficulty arises from the pairwise difference penalty  $\sum_{i<j} |\beta_i-\beta_j|$,  which contains $O(p^2)$ coupled nonsmooth terms. Consequently, highly efficient algorithms are needed for solving SRCL especially in high-dimensional case.
    
    To address this challenge, we develop a semismooth Newton-based proximal augmented Lagrangian method (N-ALM). Compared with first-order algorithms, the proposed algorithm exploits second-order information and is therefore better suited for obtaining high-accuracy solutions efficiently. 
    Related semismooth Newton augmented Lagrangian methods have been developed for several regularized optimization models, including Lasso \citep{li2018highly}, fused Lasso \citep{li2018fused}, sparse group Lasso \citep{zhang2020efficient}, OSCAR \citep{luo2019solving}, clustered Lasso \citep{lin2019efficient}, and exclusive Lasso \citep{lin2024highly}.
    Specifically, for solving \eqref{square_root_clustered_lasso_model}, our algorithm efficiently solves the resulting subproblems via a semismooth Newton method that fully exploits second-order sparsity. 
    In addition, we establish global convergence and a local linear rate for the distance to the KKT solution set under the stated conditions, with a superlinear rate when the penalty parameter tends to infinity.
    Finally, numerical experiments on real-world datasets demonstrate that N-ALM efficiently computes accurate solutions of large-scale SRCL problems and requires substantially less time than d-ADMM on the tested instances.
 
	The remaining parts of this paper are organized as follows. In Section \ref{preli}, we summarize some essential definitions and foundational results for subsequent algorithm design and numerical implementations. 
	In Section \ref{sec3}, we establish weak decomposability of the clustered Lasso penalty under suitable parameter conditions and derive a noise-level-free tuning rule.
	In Section \ref{sec3-2}, we establish finite-sample theoretical guarantees for prediction accuracy. 
	In Section \ref{sec4}, we develop the N-ALM algorithm for solving the dual formulation of the square-root clustered Lasso problem, and establish
	its convergence results under error bound condition.
	Moreover, in Section \ref{numersec}, we conduct comprehensive numerical experiments to evaluate the computational efficiency of N-ALM for the SRCL problem \eqref{square_root_clustered_lasso_model}.
	Finally, we conclude this paper with some remarks in Section \ref{con}.
    
    \textbf{Notation.} 
Throughout the paper, we use \(\operatorname{diag}(X)\) to denote the vector consisting of the diagonal entries of the matrix \(X\), and \(\operatorname{Diag}(x)\) to denote the diagonal matrix whose diagonal is given by the vector \(x\); for matrix arguments, \(\operatorname{Diag}\) denotes the corresponding block diagonal matrix. We denote by \(I_d\) the \(d\times d\) identity matrix (and omit the subscript when its size is clear) and by \(\mathbf{O}_n\) the \(n\times n\) zero matrix. For symmetric matrices \(A\) and \(B\) of the same size, \(A\preceq B\) means that \(B-A\) is positive semidefinite, and \(A\succ0\) means that \(A\) is positive definite. The largest eigenvalue of \(A\) is denoted by \(\lambda_{\max}(A)\), and \(\|X\|_2\) denotes the matrix norm induced by the Euclidean norm. For \(M\succ0\), we write \(\langle v,v'\rangle_M:=\langle v,Mv'\rangle\), \(\|v\|_M:=\sqrt{\langle v,v\rangle_M}\), and \(\operatorname{dist}_M(v,S):=\inf_{s\in S}\|v-s\|_M\). For a given matrix \(C\), \(C^\dagger\) represents its Moore--Penrose pseudoinverse. As usual, \(f^*(\cdot)\) denotes the Fenchel conjugate of an arbitrary function \(f(\cdot)\); in particular, for a norm \(p(\cdot)\), \(p^\star(\cdot)\) stands for its dual norm.
	
	\section{Preliminaries}\label{preli}

	\subsection{Basic Concepts and Some Preliminary Results }
	In this section, we review some basic concepts and definitions needed for the subsequent algorithmic developments and theoretical analysis.

    Firstly, we recall some basic concepts related to semismooth Newton methods. Let $\Psi:\mathcal O\subseteq\mathbb{R}^{n}\rightarrow\mathbb{R}^{m}$ be a locally Lipschitz continuous mapping. The  Clarke generalized Jacobian of $\Psi$ at $\nu$ is
    defined as 
    $$
     \partial\Psi(\nu) := \operatorname{conv} \left\{ \lim_{k\rightarrow\infty} \nabla\Psi(\nu_k): \nu_k\rightarrow\nu,\  \nu_k\in\mathcal D_\Psi \right\},
    $$
    where $\mathcal D_\Psi$ denotes the set of points at which $\Psi$ is differentiable.
 
	\begin{definition}[Semismoothness \citep{Miff29,qi30}]\label{def1}
		Let $\Psi:\mathcal O\subseteq\mathbb{R}^{n}\rightarrow\mathbb{R}^{m}$ be a
		locally Lipschitz continuous mapping, and let $\partial\Psi(\nu)$ denote the
		Clarke generalized Jacobian of $\Psi$ at $\nu$. 
		\begin{itemize}
			\item[(i)]	$\Psi$ is said to be semismooth at $\nu$ if it is directionally
			differentiable at $\nu$ and
			$$
			\Psi(\nu+\Delta\nu)-\Psi(\nu)-V\Delta\nu=o(\|\Delta\nu\|),
			$$
			for any $\Delta\nu\rightarrow0$ and any	$V\in\partial\Psi(\nu+\Delta\nu)$.
			
			\item[(ii)]	$\Psi$ is said to be $\gamma$-order semismooth	($\gamma>0$) at $\nu$ if it is semismooth at $\nu$ and
			$$
			\Psi(\nu+\Delta\nu)-\Psi(\nu)-V\Delta\nu=O(\|\Delta\nu\|^{1+\gamma}),
			$$
			for any $\Delta\nu\rightarrow0$ and any 	$V\in\partial\Psi(\nu+\Delta\nu)$.
			In particular, when $\gamma=1$, $\Psi$ is called strongly semismooth.
		\end{itemize}
		Moreover, $\Psi$ is said to be semismooth on $\mathcal O$ if it is
		semismooth at every point of $\mathcal O$. The definitions of
		$\gamma$-order semismoothness and strong semismoothness on $\mathcal O$ are
		analogous.
	\end{definition}
	
	For a positive definite matrix $M$, let $\operatorname{dist}_M(v, S)$ be the distance of a point $v$ to a closed convex set $S$, that is, $\operatorname{dist}_M(v, S) := \inf_{v' \in S} \|v-v'\|_M$. In particular, when $M=I$, we simply write $\operatorname{dist}(x,S):=\operatorname{dist}_I(x,S)$.

\begin{definition}[Error Bound Condition]\label{def:error_bound}
	Let $\Psi:\mathcal O\rightrightarrows\mathbb{R}^{m}$ be a set-valued mapping and let $\nu^*\in\mathcal O$ satisfy
	$0\in\Psi(\nu^*)$. If there exist constants $\varepsilon>0$ and $\kappa>0$ such that
	\begin{equation}\label{EBC}
		\operatorname{dist}\big(\nu,\Psi^{-1}(0)\big)\leq\kappa\operatorname{dist}\big(0,\Psi(\nu)\big),
	\end{equation}
	for all $\nu\in\mathcal O$ satisfying $\|\nu-\nu^*\|_2\leq\varepsilon$,
	then $\Psi$ is said to satisfy the error bound condition at $\nu^*$ with modulus $\kappa$. This condition is equivalent to the metric
	subregularity of $\Psi$ at $\nu^*$ for the origin; see,
	e.g., \cite{kruger2015error}.
\end{definition}
 Following \cite[Definition~1(a)]{zhou2017unified}, \(\Psi^{-1}:\mathbb R^m\rightrightarrows\mathcal O\) is calm at \(0\) for \(\nu^*\) if there exist \(\kappa>0\) and \(r>0\) such that
\begin{equation}\label{eq:calmness-definition-psi}
	\operatorname{dist}\bigl(\nu,\Psi^{-1}(0)\bigr)
	\le\kappa\|v\|_2,
	\qquad \|v\|_2\le r,\quad v\in\Psi(\nu),\quad \|\nu-\nu^*\|_2\le r.
\end{equation}
By \cite[Theorem~3H.3]{dontchev2009implicit}, \(\Psi^{-1}\) is calm at \(0\) for \(\nu^*\) if and only if \(\Psi\) is metrically subregular at \(\nu^*\) for the origin; by Definition~\ref{def:error_bound}, the latter is equivalent to \(\Psi\) satisfying the error bound condition at \(\nu^*\).

     Secondly, we review some basic definitions and conclusions in optimization theory \citep{Rock35}. 
     Let $f: \mathbb{R}^n\rightarrow (-\infty,+\infty]$ be a proper closed convex function. The subdifferential of $f$ at $x$, denoted as $\partial f(x)$,  is the set of  vectors $z$ that  satisfy $f(y)\geq f(x)+\langle z,y-x \rangle$ for an arbitrary $y \in \mathbb{R}^n$. Such a vector $z$ is referred to as a subgradient of  $f$ at $x$. 
     The Fenchel conjugate of $f$ at $x^*$ is defined by $f^*(x^*):=\sup_{x\in\mathbb{R}^n}\{\langle x, x^*\rangle-f(x)\}$.
     The Moreau envelope function of $f$ with parameter $t >0$, denoted as $\Phi_{t f}(x)$, is defined as
     \begin{equation}\label{moreau}
     	\Phi_{t f}(x) :=\min _{y \in \mathbb{R}^{n}}\Big\{f(y)+\frac{1}{2 t}\|y-x\|_2^{2}\Big\} \quad \forall x \in \mathbb{R}^n.
     \end{equation}
     It is well known that the Moreau envelope $\Phi_{tf}(\cdot)$ is continuously differentiable on $\mathbb R^n$  \citep{rockafellar1998variational}.
     The proximal mapping of $f$ with $t>0$, denoted as $\operatorname{Prox}_{tf}(x)$, is defined by
     \begin{equation}\label{proxf}
     	\operatorname{Prox}_{tf}(x) :=\underset{y \in \mathbb{R}^{n}}{\operatorname{argmin}}\Big\{f(y)+\frac{1}{2t}\|y-x\|_2^{2}\Big\} \quad \forall x \in \mathbb{R}^n.
     \end{equation}
     The following Moreau's identity \cite[Theorem 24.5]{bauschke2017} is essential in the computation of proximal mapping and  subsequent analysis:
     \begin{equation}\label{moreauid}
     	\text{Prox}_{tf}(x)+t\text{Prox}_{f^{*}/t}(x/t)=x.
     \end{equation}
	The function $f(\cdot)$ is said to be positively homogeneous of degree one if $f(\lambda x) = \lambda f(x)$ for all $x\in\mathbb{R}^n$ and $\lambda>0$.
	Moreover, then $f(\cdot)$ is positively homogeneous of degree one if and only if its Fenchel conjugate $f^*(\cdot)$ reduces to the indicator function of the closed convex set
	$$
	\mathcal{C} := \big\{ x^* \in \mathbb{R}^n : \langle x, x^* \rangle \le f(x), \; \forall x \in \mathbb{R}^n \big\}.
	$$
	If $f(0) = 0$, then $\mathcal{C} = \partial f(0)$.	
	A notable example is $f(x) = \mu\|x\|_1$ with a scalar $\mu > 0$, then its conjugate function is $f^*(z) = \delta_{B_\infty^{(\mu)}}(z)$ where $B_\infty^{(\mu)} := \{ z \mid \|z\|_\infty \le \mu \}$. Using Moreau identity \eqref{moreauid}, the proximal mapping of $f(\cdot)$ is given by
	\begin{equation}\label{proxl1}
	\operatorname{Prox}_{\mu\|\cdot\|_1}(z) = z - \Pi_{B_{\infty}^{(\mu)}}(z) \quad \text{with} \quad
	\Pi_{B_{\infty}^{(\mu)}}(z)_i = \left\{\begin{array}{ll}
		z_i, & \text{ if } |z_i| \le \mu, \\[2mm]
		\text{sign}(z_i)\mu, & \text{ if } |z_i| > \mu.
	\end{array}\right.
	\end{equation}
	where $\text{sign}(\cdot)$ is the componentwise sign function.

	Another popular example is $f(x) = \mu\|x\|_2$ with a scalar $\mu > 0$, then its conjugate function is $f^*(z) = \delta_{B_2^{(\mu)}}(z)$ where $B_2^{(\mu)} := \{ z \mid \|z\|_2 \le \mu \}$. Using Moreau identity \eqref{moreauid}, the proximal mapping of $f(\cdot)$ is given by
	$$
	\operatorname{Prox}_{\mu\|\cdot\|_2}(z) = z - \Pi_{B_{2}^{(\mu)}}(z) \quad \text{with} \quad
	\Pi_{B_{2}^{(\mu)}}(z) = \left\{\begin{array}{ll}
		z, & \text{ if } \|z\|_2 \le \mu, \\[2mm]
		\frac{\mu}{\|z\|_2} z, & \text{ if } \|z\|_2 > \mu.
	\end{array}\right.
	$$
	By Lemma 1 of \cite{yuan2018efficient}, the proximal mapping	$\operatorname{Prox}_{\mu\|\cdot\|_2}(z)$ is strongly semismooth for any	$\mu>0$.
	For any given $\vartheta\in\mathbb{R}^n$, the Clarke generalized Jacobian \citep{Clarke1990,qi30} of the proximal mapping $\operatorname{Prox}_{\mu\|\cdot\|_2}(\cdot)$ at point $\vartheta$ is given by
	\begin{equation}\label{l2clarke}
	\partial \operatorname{Prox}_{\mu\|\cdot\|_{2}}(\vartheta) = \left\{\begin{array}{ll}
		\left( 1 - \frac{\mu}{\|\vartheta\|_2} \right) I_n + \mu \frac{\vartheta \vartheta^\top}{\|\vartheta\|_2^3}, & \text{if } \|\vartheta\|_2 > \mu, \\[3mm]
		\left\{ \kappa \frac{\vartheta \vartheta^\top}{\mu^2} \mid 0 \le \kappa \le 1 \right\}, & \text{if } \|\vartheta\|_2 = \mu, \\[3mm]
		\mathbf{O}_n, & \text{if } \|\vartheta\|_2 < \mu,
	\end{array}\right.
	\end{equation}
	where $I_n$ denotes the $n\times n$ identity matrix and $\mathbf{O}_n$ denotes the $n\times n$ zero matrix.
	For practical implementation, we select the following element from the Clarke generalized Jacobian $\partial \operatorname{Prox}_{\mu\|\cdot\|_{2}}(\vartheta)$:
	$$
	\hat{\Theta} = \left\{\begin{array}{ll}
		\left( 1 - \frac{\mu}{\|\vartheta\|_2} \right) I_n + \mu \frac{\vartheta \vartheta^\top}{\|\vartheta\|_2^3}, & \text{if } \|\vartheta\|_2 > \mu, \\[2mm]
		\mathbf{O}_n, & \text{if } \|\vartheta\|_2 \le \mu,
	\end{array}\right.
	$$	
	which is symmetric and positive definite when $\|\vartheta\|_2 > \mu$ and symmetric and positive semidefinite otherwise. 
	
	Finally, we briefly recall some basic concepts from convex analysis \cite{Rock35}, which will be used to characterize the dual norm of the clustered Lasso penalty in the prediction error bound analysis.
	For a nonempty closed convex set $\mathcal{C}\subseteq\mathbb{R}^{p}$, its polar set is defined as $\mathcal{C}^{\circ}=\{z\in\mathbb{R}^{p}\mid\langle z,x\rangle\leq1, \ \forall x\in\mathcal{C}\}$. The support function of $\mathcal{C}$ is defined as $\sigma_{\mathcal{C}}(x)=\sup_{y\in\mathcal{C}}\langle x,y\rangle$, and the indicator function of $\mathcal{C}$, denoted $\delta_{\mathcal{C}}(x)$, satisfies $\delta_{\mathcal{C}}(x)=0$ for $x\in\mathcal{C}$, and $+\infty$ otherwise.
	Clearly, these two functions are Fenchel conjugates of each other, i.e., $\sigma_{\mathcal{C}}(x)=\delta_{\mathcal{C}}^*(x)$. 
	The gauge function (or Minkowski functional) associated with $\mathcal{C}$ is defined as $\gamma_{\mathcal{C}}(x)=\inf\{t>0\mid x\in t\mathcal{C}\}$.
	Suppose $\mathcal{C}$ is a nonempty closed convex set containing the origin. It follows from \cite[Theorem~14.5]{Rock35} that the support function of $\mathcal{C}^{\circ}$ is the gauge function of $\mathcal{C}$, i.e., $\sigma_{\mathcal{C}^{\circ}}(\cdot)=\gamma_{\mathcal{C}}(\cdot)$.
	Moreover, if $\mathcal C$ is bounded and symmetric and contains the origin in its interior, its gauge function is a norm.
	For instance, the $\ell_1$-norm, $\ell_2$-norm, and $\ell_\infty$-norm unit balls, whose gauge functions are exactly $\|\cdot\|_1$, $\|\cdot\|_2$, and $\|\cdot\|_\infty$, respectively.
	In particular, if $\mathcal{C}$ is the unit ball of a norm $p(\cdot)$, then its support function coincides with the dual norm, i.e., $\sigma_{\mathcal{C}}(z)=p^\star(z)$, and the unit ball of the dual norm is the polar set $\mathcal{C}^{\circ}$, i.e., $\mathcal{C}^{\circ}=\{z\mid p^\star(z)\leq 1\}$.
	
   Suppose $z\in\partial p(\beta)$. By subdifferential definition, we have $p(x)\ge p(\beta)+\langle z,x-\beta\rangle$, $\forall x$. Since $p$ is a norm (positively homogeneous), taking $x=0$ and $x=2\beta$ yields $\langle z,\beta\rangle=p(\beta)$. Substitute it back, we get $\langle z,x\rangle \le p(x)$,  $\forall x$.
   By definition of dual norm $p^\star(z)=\sup_{p(x)\le1}\langle z,x\rangle$, we get $p^\star(z)\le 1$, that is to say
   \begin{equation}\label{ppstar}
    z\in\partial p(\beta)\Longrightarrow p^\star(z)\leq1 .
   \end{equation}
   When $\beta$ is nonzero, we actually have $p^\star(z)=1$. At $\beta=0$ alone, the subdifferential $\partial p(0)$ consists of all $z$ satisfying $p^\star(z)\le1$.
	This standard characterization of the subdifferential of a norm follows from 
	\cite[Theorem 23.5]{Rock35} or
	\cite[Example 8.10]{rockafellar1998variational}.
	
\subsection{Preliminary Results on the Proximal Mapping for Clustered Lasso Penalty}
   	In this section, we briefly review several useful results on the proximal mapping of the clustered Lasso penalty
   \begin{equation}\label{defcl}
   p(x)=\lambda_1\|x\|_1+\lambda_2\sum_{1\le i<j\le p}|x_i-x_j|,  \qquad x\in\mathbb R^p,
   \end{equation}
   and its computable generalized Jacobian established in \cite{lin2019efficient}. 
   Clearly, the clustered Lasso penalty defined in \eqref{defcl} is a norm on $\mathbb{R}^p$ provided that $\lambda_1>0$ and $\lambda_2\geq 0$.
   Let $\mathcal D:=\{x\in\mathbb R^p\mid Bx\ge0\}$, where $B\in\mathbb R^{(p-1)\times p}$ is the adjacent-difference matrix satisfying
   $Bx=(x_1-x_2,\ldots,x_{p-1}-x_p)^\top$.
   For any $x\in\mathbb R^p$, let $x^\downarrow=P_xx$ be sorted in nonincreasing order, and define $w_i=p-2i+1$ for $i=1,\ldots,p$.
   It follows from \cite{lin2019efficient} that
   \[
   \mathcal S_{\lambda_2}(x):=P_x^\top\Pi_{\mathcal D}(P_xx-\lambda_2w),
   \qquad
   \operatorname{Prox}_p(x)=\operatorname{Prox}_{\lambda_1\|\cdot\|_1}(\mathcal S_{\lambda_2}(x)).
   \]
   The projection $\Pi_{\mathcal D}$ is an isotonic regression problem and can be computed by the pool-adjacent-violators (PAV) algorithm \citep{best1990active}.

   We next define the generalized derivative mapping used to compute Newton matrices. For $v\in\mathbb R^p$, set
   \[
   \mathcal I(v):=\{i\in\{1,\ldots,p-1\}:(B\Pi_{\mathcal D}(v))_i=0\}.
   \]
   Let $\ell(v)\in\mathbb R^{p-1}$ be the unique projection multiplier satisfying
   \[
   \Pi_{\mathcal D}(v)-v+B^\top\ell(v)=0,\quad
   \ell(v)\le0,\quad \ell_i(v)=0\quad(i\notin\mathcal I(v)).
   \]
   Define $\mathcal K(v):=\{K:\operatorname{supp}(\ell(v))\subseteq K\subseteq\mathcal I(v)\}$, and let
   \[
   \mathcal Q(x):=\left\{P_x^\top\left[I_p-B_K^\top(B_KB_K^\top)^{-1}B_K\right]P_x:
   K\in\mathcal K(P_xx-\lambda_2w)\right\}.
   \]
   Here $B_K$ contains the rows indexed by $K$; when $K=\varnothing$, the bracketed matrix is $I_p$.
   Following \cite[Eq.~(2.4) and Proposition~2.10]{lin2019efficient}, define the full multifunction
   \begin{equation}\label{M}
   \mathcal M(x):=\left\{\Theta Q:
   \Theta\in\partial_B\operatorname{Prox}_{\lambda_1\|\cdot\|_1}(\mathcal S_{\lambda_2}(x)),\ Q\in\mathcal Q(x)\right\},
   \end{equation}
   where $\partial_B$ denotes the $B$-subdifferential. This mapping is nonempty, compact, and upper semicontinuous. For every $M\in\mathcal M(x)$, both $M$ and $I_p-M$ are symmetric positive semidefinite. Moreover, there exists a neighborhood $\mathcal X$ of $x$ such that
   \begin{equation}\label{prox_semismooth}
   \operatorname{Prox}_p(u)-\operatorname{Prox}_p(x)-M(u-x)=0,
   \qquad u\in\mathcal X,\quad M\in\mathcal M(u).
   \end{equation}
   Thus $\operatorname{Prox}_p$ is strongly semismooth with respect to $\mathcal M$; see \cite[Remark~2.12]{lin2019efficient}. These properties concern the full multifunction, rather than a particular single-valued selection, and $\mathcal M$ is not identified with the Clarke generalized Jacobian.

   A computable element of $\mathcal M(x)$ is obtained by taking $\mathcal I:=\mathcal I(P_xx-\lambda_2w)$ and
   \[
   \Sigma=\operatorname{Diag}(a)\in\mathbb R^{(p-1)\times(p-1)},\quad a_i=\mathbf 1_{\{i\in\mathcal I\}},
   \qquad
   \Theta=\operatorname{Diag}(\theta),\quad
   \theta_i=\mathbf 1_{\{|(\mathcal S_{\lambda_2}(x))_i|>\lambda_1\}}.
   \]
   Specifically, $M=\Theta P_x^\top\Gamma P_x\in\mathcal M(x)$, where
   \begin{equation}\label{Gamma}
   \Gamma=I_p-B^\top(\Sigma BB^\top\Sigma)^\dagger B.
   \end{equation}
   The matrix $B$ used here is distinct from the full pairwise-difference matrix $D$ introduced in Lemma~\ref{lem:dual_norm}.

   To compute $\Gamma$ efficiently, we use the structural decomposition established in \cite[Proposition~2.8]{lin2019efficient}.
   More precisely, suppose that
   $ \Sigma=\operatorname{Diag}(\Lambda_1,\ldots,\Lambda_N)$,
   where each block $\Lambda_i$ is either $\mathbf O_{p_i}$ or $I_{p_i}$, and any two consecutive blocks are of different types. Let
   $ J:=\{\,i\mid\Lambda_i=I_{p_i}\,\}$. 
   Then
   \begin{equation}\label{gammat}
   	\Gamma  	= 	H+U_JU_J^\top,
   \end{equation}
   where $H$ is a block diagonal matrix and $U_J$ consists of the columns of a sparse matrix $U$ indexed by $J$. The explicit constructions of $H$ and $U$ can be found in \citep[Proposition~2.8]{lin2019efficient}.

    \section{Tuning Parameter Selection Analysis}\label{sec3}
   \subsection{Weak Decomposability of the Clustered Lasso Penalty}
    In this subsection, we show that the clustered Lasso penalty defined in \eqref{defcl} is weakly decomposable with respect to $S$ when $\lambda_1>\lambda_2|S|$. We begin by recalling the definition of a weakly decomposable norm of \cite{van2014weakly}.
    
    \begin{definition}[Weakly Decomposable Norm \citep{van2014weakly}]\label{def:weak_decompo-norm}
    	A norm $\mathcal{P}$ on $\mathbb{R}^p$ is said to be weakly decomposable with respect to an index set $S\subseteq\{1,\ldots,p\}$ if there exists a norm $\mathcal{P}^{c}$ on $\mathbb{R}^{|S^c|}$ such that
    	$$
    	\mathcal{P}(\beta)\geq \mathcal{P}(\beta_S) + \mathcal{P}^{c}(\beta_{S^c}), \qquad	\forall\,\beta\in\mathbb{R}^p.
    	$$
    	The set $S$ is called an allowed set if $\mathcal{P}$ is weakly decomposable with respect to $S$, where $\beta_S$ denotes the vector obtained from $\beta$ by setting all components outside $S$ to zero.
    \end{definition}
    
	\begin{theorem}\label{wdtheorem}
		The clustered Lasso penalty defined in \eqref{defcl}, i.e.,
		$$
		p(\beta)=\lambda_1\|\beta\|_1+\lambda_2\sum_{i<j}|\beta_i-\beta_j|, \qquad	\forall\,\beta\in\mathbb{ R}^p,
		$$
		is a weakly decomposable norm with respect to an index set $S\subseteq\{1,2,\ldots,p\}$ with $s=|S|$ provided that $\lambda_1>\lambda_2s$.
	\end{theorem}
	
	 \begin{proof}
	 	Decompose $p(\beta)=p_1(\beta)+p_2(\beta)$, where $ p_1(\beta) := \lambda_1 \|\beta\|_1 $ and $p_2(\beta) := \lambda_2 \sum_{i < j} |\beta_i - \beta_j| $. 
	 	Firstly,  for an index set $S \subset \{1, \ldots, p\}$, by the definition of the $\ell_1$-norm, we get that $p_1(\beta)$ can be expressed as
		$$
		p_1(\beta) = \lambda_1 \|\beta\|_1 =   \lambda_1 \left( \sum\limits_{i\in S} |\beta_i| + \sum\limits_{i\in S^c} |\beta_i| \right) = \lambda_1 \|\beta_{S}\|_1 + \lambda_1 \|\beta_{S^c}\|_1,
		$$
		where  $\beta_S=\big(\beta_i\mathbf{1}_{\{i\in S\}}\big)_{i=1}^p\in\mathbb{R}^p$, i.e., $\beta_S$ retains the components of $\beta$ indexed by $S$ and sets all other components to zero, and $\beta_{S^c}\in\mathbb{R}^{|S^c|}$ denote the subvector of $\beta$ corresponding to the index set $S^c$.
		Let $p_1(\beta_{S}) = \lambda_1 \|\beta_{S}\|_1 $ and $p_1^{c}(\beta_{S^c}) = \lambda_1 \|\beta_{S^c}\|_1 $. We can derive that
		\begin{equation}\label{pb1}
		p_1(\beta) = p_1(\beta_{S}) + p_1^{c}(\beta_{{S}^c}).
		\end{equation}
		Secondly, for the given index set $S$, we let
		$$
			p_2(\beta_S):=\lambda_2\sum_{i<j}|(\beta_S)_i-(\beta_S)_j|=\lambda_2\sum_{i,j\in S,i<j}|\beta_i-\beta_j|+\lambda_2(p-s)\|\beta_S\|_1,
		$$
		and
		$$
		p_2^{c}(\beta_{S^c}):=\lambda_2\sum_{i,j\in S^c,i<j}|\beta_i-\beta_j|.
		$$
		We know that, for any $i\in S$ and $j\in S^c$, it holds that $|\beta_i-\beta_j|\geq |\beta_i|-|\beta_j|$.
		Therefore, the cross terms satisfy
		$$
			\sum_{i\in S}\sum_{j\in S^c}|\beta_i-\beta_j|
			\geq\sum_{i\in S}\sum_{j\in S^c}\left(|\beta_i|-|\beta_j|\right)=(p-s)\|\beta_S\|_1-s\|\beta_{S^c}\|_1 .
		$$
		Consequently, it holds that
		$$
		\begin{aligned}
			p_2(\beta)
			=&
			\lambda_2\sum_{i,j\in S,i<j}|\beta_i-\beta_j|+\lambda_2\sum_{i,j\in S^c,i<j}
			|\beta_i-\beta_j|+\lambda_2\sum_{i\in S}\sum_{j\in S^c}	|\beta_i-\beta_j|\\[2mm]
			\geq&
			\lambda_2\sum_{i,j\in S,i<j}|\beta_i-\beta_j|+\lambda_2\sum_{i,j\in S^c,i<j}|\beta_i-\beta_j|+\lambda_2(p-s)\|\beta_S\|_1
			-\lambda_2s\|\beta_{S^c}\|_1.
		\end{aligned}
		$$
		Recall the definitions of $p_2(\beta_S)$ and $p_2^{c}(\beta_{S^c})$, then this inequality implies that
		\begin{equation}\label{pb2}
		p_2(\beta)	\geq p_2(\beta_S)+p_2^{c}(\beta_{S^c})	-\lambda_2s\|\beta_{S^c}\|_1.
		\end{equation}
		Combining the inequalities \eqref{pb1} and \eqref{pb2}, we get
		$$
		\begin{aligned}
			p(\beta)
			=p_1(\beta)+p_2(\beta)
			\geq	\lambda_1\|\beta_S\|_1	+p_2(\beta_S)+(\lambda_1-\lambda_2s)\|\beta_{S^c}\|_1+p_2^{c}(\beta_{S^c})
			=p(\beta_S)+p^{c}(\beta_{S^c}),
		\end{aligned}
		$$
		where
		\begin{equation}\label{defpcc}
		p(\beta_S):=p_1(\beta_S)+p_2(\beta_S)\quad\text{and}\quad
		p^{c}(\beta_{S^c}):=(\lambda_1-\lambda_2s)\|\beta_{S^c}\|_1+	\lambda_2	\sum_{i,j\in S^c,i<j}|\beta_i-\beta_j|.
		\end{equation}
		
		It remains to verify that $p^{c}$ is a norm on	$\mathbb R^{|S^c|}$. Suppose that $\lambda_1>\lambda_2s$.
		Then the coefficient of the $\ell_1$-term in $p^{c}$ is strictly	positive.
		For any $\beta_{S^c}\neq0$, we have $p^{c}(\beta_{S^c})>0$,	and hence the positive definiteness holds. 
		Moreover, for any $t\in\mathbb R$, it holds that
		$$
			p^{c}(t\beta_{S^c})=(\lambda_1-\lambda_2s)\|t\beta_{S^c}\|_1	+\lambda_2\sum_{i,j\in S^c,i<j}	|t\beta_i-t\beta_j|	=|t|p^{c}(\beta_{S^c}),
		$$
		which proves the positive homogeneity.		Finally, for any $\beta_{S^c},\gamma_{S^c}\in\mathbb R^{|S^c|}$,
		the triangle inequality of the $\ell_1$-norm and the absolute value function yield
		$$
		\begin{aligned}
			p^{c}(\beta_{S^c}+\gamma_{S^c})
			\leq&(\lambda_1-\lambda_2s)	\|\beta_{S^c}\|_1	+\lambda_2\sum_{i,j\in S^c,i<j}	|\beta_i-\beta_j|+(\lambda_1-\lambda_2s)
			\|\gamma_{S^c}\|_1	+\lambda_2\sum_{i,j\in S^c,i<j}|\gamma_i-\gamma_j|\\
			=&p^{c}(\beta_{S^c})+p^{c}(\gamma_{S^c}).
		\end{aligned}
		$$
		Therefore, $p^{c}$ is a norm on $\mathbb R^{|S^c|}$.
		Hence, under the condition $\lambda_1>\lambda_2s$, the clustered Lasso	penalty is a weakly decomposable norm with respect to the index set $S$.
		\qed
\end{proof}

\subsection{Noise-Level-Free Choice of the Tuning Parameter}
One important advantage of the proposed SRCL method \eqref{square_root_clustered_lasso_model} is that its tuning parameters $\lambda_1$ and $\lambda_2$ can be selected without knowledge of the unknown noise level $\sigma$ appearing in \eqref{lineareg}. This pivotal property is consistent with the square-root regularization framework with weakly decomposable norm penalties studied in \cite[Theorem 9]{stucky2017sharp}, where sharp oracle inequalities are established under tuning parameters determined by suitably normalized noise quantities.
In fact, this pivotal property can also be found in several square-root-based estimators, including the square-root Lasso \citep{belloni2011square}, the square-root group Lasso \citep{belloni2011square,yuan2006model}, and the square-root SLOPE \citep{bogdan2015slope}.

However, the existing estimators typically involve only a single tuning parameter. In contrast, the clustered Lasso penalty considered in \eqref{square_root_clustered_lasso_model} involves two tuning parameters $\lambda_1$ and $\lambda_2$, corresponding to the sparsity-inducing and clustering components, respectively. 
In what follows, we establish the pivotal choice of these two tuning parameters, and show that they can be theoretically specified without knowledge of the unknown noise level $\sigma$.

\begin{assumption}\label{assumption_noise}
	The noise vector $\epsilon=(\epsilon_1,\ldots,\epsilon_n)^{\top}$ has independent standard Gaussian components and is independent of the random design matrix $X$. 
\end{assumption}
\begin{assumption}\label{assumption_design}
	The columns of the design matrix $X=(X_1,\ldots,X_p)\in\mathbb{R}^{n\times p}$ are normalized such that
	$$
		\|X_j\|_2=\sqrt{n},\qquad j=1,\ldots,p.
	$$
\end{assumption}

For convenience, we define
\begin{equation}\label{defpr}
	p_{\rho}(\beta)	:=\|\beta\|_1+\rho\sum_{1\leq i<j\leq p}|\beta_i-\beta_j|,
\end{equation}
where $\rho>0$ is a fixed scalar.  Then, from \eqref{square_root_clustered_lasso_model}, the proposed SRCL estimator $\widehat{\beta}$ can be written in the unified square-root penalized form
\begin{equation}\label{SRCL_unified}
	\widehat{\beta}\in\arg\min_{\beta\in\mathbb{R}^p}	\Big\{\frac{1}{\sqrt{n}}\|y-X\beta\|_2+\lambda p_{\rho}(\beta)\Big\}
\end{equation}
provided that $\lambda_1=\lambda$ and $\lambda_2=\rho\lambda$.
Clearly, $p_\rho(\beta)=p(\beta)/\lambda$ is a norm. By Theorem~\ref{wdtheorem}, it is weakly decomposable with respect to an index set $S$ whenever $\rho|S|<1$.

\begin{lemma}\label{lem:dual_norm}
	Let $D\in\mathbb{R}^{p(p-1)/2\times p}$ be a pairwise difference matrix such that
	$$
		\sum_{1\leq i<j\leq p}|\beta_i-\beta_j|	=\|D\beta\|_1,
    $$
	and hence $	p_{\rho}(\beta)=\|\beta\|_1+\rho\|D\beta\|_1$. 	The dual norm of $p_{\rho}(\cdot)$ is given by
	\begin{equation}\label{dual_norm_explicit}
		p_{\rho}^\star(u)=\min_{z\in\mathbb{R}^{p(p-1)/2}}\left\{\max\Big\{\|u-D^{\top}z\|_{\infty},\frac{\|z\|_{\infty}}{\rho}\Big\}\right\},	\qquad	u\in\mathbb{R}^p.
	\end{equation}
\end{lemma}

\begin{proof}
	From the definition of dual norm, it gets that
	$$
		\|\beta\|_1	=\max_{\|a\|_\infty\leq 1}a^\top\beta,\quad\text{and}\quad \rho\|D\beta\|_1	=	\max_{\|z\|_\infty\leq\rho}z^\top D\beta.
	$$
	Therefore, from the mathematical expression of $p_\rho(\beta)$, it yields that 
	$$
		p_{\rho}(\beta)=\max_{\|a\|_\infty\leq 1}a^\top\beta+\max_{\|z\|_\infty\leq\rho}z^\top D\beta
		=\max_{\substack{\|a\|_\infty\leq 1\\ \|z\|_\infty\leq\rho}}\Big\{	(a+D^\top z)^\top\beta\Big\}.
	$$
	For convenience, we define $\mathcal{C}_{\rho}:=\left\{	a+D^\top z\mid \|a\|_\infty\leq1,\,\|z\|_\infty\leq\rho\right\}$.
	Hence, $p_{\rho}$ can be represented as the support function of	$\mathcal{C}_{\rho}$, namely,
	$$
		p_{\rho}(\beta)	=\sup_{u\in\mathcal{C}_{\rho}}u^\top\beta=:\sigma_{\mathcal{C}_\rho}(\beta).
	$$
	According to the definition of the polar set reviewed in Section \ref{preli}, we have
    $$
		\mathcal{C}_{\rho}^{\circ}=\left\{\beta \mid u^{\top}\beta\leq1,\ \forall u\in\mathcal{C}_{\rho}\right\}  
		=\Big\{\beta\mid \sup_{u\in\mathcal{C}_{\rho}}u^{\top}\beta\leq1\Big\}= \Big\{\beta\mid p_{\rho}(\beta)\leq1 \Big\}.
	$$
    Hence, by the definition of the dual norm, we have
    \begin{equation}\label{dunormp}
    	p_{\rho}^{\star}(u)  =   \sup_{p_{\rho}(\beta)\leq1}u^\top\beta 	=  	\sigma_{\mathcal{C}_{\rho}^{\circ}}(u).
    \end{equation}
    By Theorem 14.5 in \cite{Rock35},   the support function of $\mathcal{C}_{\rho}^{\circ}$ is the gauge
    function of $\mathcal{C}_{\rho}$. Therefore,
    $$
    	p_{\rho}^{\star}(u)=\sigma_{\mathcal{C}_{\rho}^{\circ}}(u)
    	=\gamma_{\mathcal{C}_{\rho}}(u)	= 	\inf\Big\{	t>0\mid u\in t\mathcal{C}_{\rho}   	\Big\}.
    $$
    Moreover, for $t>0$, the condition $u\in t\mathcal{C}_{\rho}$ is equivalent to
	the existence of $a$ and $z$ such that	$u:=a+D^\top z$ with $\|a\|_\infty\leq t$ and $\|z\|_\infty\leq\rho t$. 
	Eliminating $a$ by setting	$a=u-D^\top z$,
	we obtain
	$$
		p_{\rho}^\star(u)=\inf_{z\in\mathbb{R}^{p(p-1)/2}}\max\left\{\|u-D^\top z\|_\infty,	\frac{\|z\|_\infty}{\rho}\right\}.
	$$
	Since the objective function is continuous and coercive with respect to
	$z$, the infimum is attained. Therefore, the result \eqref{dual_norm_explicit} holds. This completes the proof.\qed
\end{proof}

From the optimality condition of problem \eqref{SRCL_unified}, we know that any optimal solution $\widehat{\beta}$ satisfies
$$
0\in -\frac{X^{\top}(y-X\widehat{\beta})}{\sqrt n\|y-X\widehat{\beta}\|_2}+\lambda\partial p_\rho(\widehat{\beta}),
$$
provided that $y-X\widehat{\beta}\neq0$. Hence, there exists $z\in\partial p_\rho(\widehat{\beta})$ such that
$$
\frac{X^{\top}(y-X\widehat{\beta})}{\sqrt n\|y-X\widehat{\beta}\|_2}=\lambda z.
$$
Since the subgradient of a norm satisfies the relation \eqref{ppstar}, i.e., $z\in\partial p_\rho(\beta)\Longrightarrow p_\rho^*(z)\leq1$,
we obtain the following dual feasibility condition
\begin{equation}\label{dual_feasibility}
	p_\rho^\star\left(\frac{X^{\top}(y-X\widehat{\beta})}{\sqrt n\|y-X\widehat{\beta}\|_2}\right)
	\leq\lambda .
\end{equation}
However, the residual vector $y-X\widehat{\beta}$ depends on the unknown estimator
$\widehat{\beta}$ and therefore cannot be used directly for choosing the tuning
parameter. Following the pivotal tuning principle of square-root regularization
\cite{belloni2011square,stucky2017sharp}, we instead consider the corresponding
quantity involving the true noise vector.

\begin{theorem}\label{thm:pivotal_tuning}
	Suppose that Assumptions \ref{assumption_noise}--\ref{assumption_design} hold.
	Let $p_\rho(\cdot)$ be the norm defined in \eqref{defpr},
	and let $p_\rho^\star(\cdot)$ denote its dual norm.
	Define
	\begin{equation}\label{pivotal_quantity}
		\Lambda	:=p_\rho^\star\left(\frac{X^\top\epsilon}{\sqrt n\|\epsilon\|_2}\right).
	\end{equation}
	Then the conditional distribution of $\Lambda$ given $X$ does not depend on the unknown noise level
	$\sigma$.	Moreover, let $\alpha\in(0,1)$, and $q_{1-\alpha}=q_{1-\alpha}(X)$ be the conditional $(1-\alpha)$-quantile of $\Lambda$ given $X$.
	If	
	\begin{equation}\label{lambda_condition}
		\lambda\geq c q_{1-\alpha},
	\end{equation}
	for some constant $c>1$, then
	\begin{equation}\label{regularization_event}
		\mathbb{P}\left(\Lambda\leq\frac{\lambda}{c}\right)\geq1-\alpha.
	\end{equation}
	Consequently, the tuning parameter $\lambda$ can be selected independently of
	the unknown noise level $\sigma$.
	In particular, for a fixed $\rho$, the penalty parameters
	$\lambda_1=\lambda$ and $\lambda_2=\rho\lambda$ are also independent of $\sigma$.
\end{theorem}
\begin{proof}
	Under Assumption~\ref{assumption_noise}, the error vector in \eqref{lineareg} can be written as
	$y=X\beta^0+\sigma\epsilon$,
	where $\epsilon$ has independent standard Gaussian components.	
	All probability statements in the statistical analysis are understood conditionally on the observed design $X$ satisfying the stated assumptions. Since $\epsilon$ is independent of $X$, the same Gaussian noise law applies conditionally. Since the dual norm $p_\rho^\star(\cdot)$ is positively homogeneous, we have
	$$
	p_\rho^\star\left(\frac{X^{\top}(y-X{\beta^0})}{\sqrt n\|y-X{\beta^0}\|_2}\right)=p_\rho^\star\left(\frac{X^\top(\sigma\epsilon)}{\sqrt n\|\sigma\epsilon\|_2}\right)=p_\rho^\star\left(\frac{X^\top\epsilon}{\sqrt n\|\epsilon\|_2}\right).
	$$
	Hence, the random variable $\Lambda$ defined in \eqref{pivotal_quantity} is invariant with respect to the unknown scale parameter $\sigma$.	By independence of $\epsilon$ and $X$, its conditional distribution given $X$ and its conditional quantiles are independent of $\sigma$. The quantile condition therefore yields the stated probability bound conditionally on $X$, and hence also unconditionally.
	By the definition of $q_{1-\alpha}$, we get
	$$
	\mathbb P(\Lambda\leq q_{1-\alpha})\geq1-\alpha .
	$$
	If $\lambda\geq cq_{1-\alpha}$, then
	$$
	\mathbb P\left(\Lambda\leq\frac{\lambda}{c}\right)\geq1-\alpha .
	$$
	Therefore, $\lambda$ can be chosen without knowledge of $\sigma$.
	For a fixed $\rho$, the choices$\lambda_1=\lambda$ and $\lambda_2=\rho\lambda$	are also independent of $\sigma$.\qed
\end{proof}
    
As established in Theorem \ref{thm:pivotal_tuning}, the proposed SRCL estimator $\widehat\beta$ permits tuning‑parameter selection without requiring prior information on the noise level. Furthermore, under $\rho|S|<1$, the weak decomposability of $p_{\rho}$ with respect to $S$ renders the clustered Lasso penalty well‑suited to the square‑root regularization framework from \cite{stucky2017sharp}.

\section{Finite-sample Prediction Error Bound Analysis}\label{sec3-2}
\setcounter{equation}{0}
We next establish finite-sample error bounds for the proposed SRCL estimator $\widehat\beta$ defined in  \eqref{SRCL_unified}. 
The analysis is based on the weak decomposability of the penalty $p(\beta)$ (or $p_\rho(\beta)$) established in Theorem \ref{wdtheorem} and the pivotal choice of the tuning parameter established in Theorem \ref{thm:pivotal_tuning}.
For the nonzero-signal case, define
$$
S:=\{j:\beta_j^0\neq 0\},\qquad s:=|S|,
$$
where $\beta^0$ is the true regression coefficient vector in \eqref{lineareg}. Throughout the compatibility-based analysis below, assume $1\le s\le p$. 
We assume $\lambda_1>\lambda_2s$, or equivalently $\rho s<1$, so that $S$ is allowed by Theorem~\ref{wdtheorem}, recalling that $\lambda_1=\lambda$ and $\lambda_2=\rho\lambda$.

For any $\Delta\in\mathbb{R}^p$, recall that $\Delta_S$ denotes the vector obtained by retaining the components of $\Delta$ indexed by $S$ and setting all other components to zero, while $\Delta_{S^c}\in\mathbb{R}^{|S^c|}$ denotes the subvector of $\Delta$ corresponding to the index set $S^c$.
Let $p_{\rho}^c$ denote the norm on $\mathbb{R}^{|S^c|}$ associated with the weak decomposability of $p_\rho$ with respect to $S$. 
Then, from Definition \ref{def:weak_decompo-norm}, it holds that
$$
p_\rho(\Delta)\geq p_\rho(\Delta_S)+p^c_{\rho}(\Delta_{S^c}),\qquad \forall\Delta\in\mathbb{R}^p.
$$
Following the framework of \cite{van2014weakly,stucky2017sharp} for
weakly decomposable norms, for a constant $L>1$, we define the
cone associated with the penalty $p_\rho$ by
$$
\mathcal{C}_\rho(L,S):=\Big\{\Delta\in\mathbb{R}^p\mid p^c_{\rho}(\Delta_{S^c})\leq L p_\rho(\Delta_S)\Big\}.
$$
We then define the corresponding $p_\rho$-compatibility constant as
\begin{equation}\label{kappals}
\kappa_\rho(L,S):=\inf_{\substack{\Delta\in\mathcal{C}_\rho(L,S)\\
		p_\rho(\Delta_S)>0}} \left\{\frac{\sqrt{s}\|X\Delta\|_2}{\sqrt{n}p_\rho(\Delta_S)}\right\},
\end{equation}
and assume that $\kappa_\rho(L,S)>0$.

The following lemma shows that, under the pivotal event, the estimation error $\Delta$ satisfies an $(L,S)$-cone condition. Its proof follows the standard cone argument for weakly decomposable regularization penalties; see \cite{van2014weakly,stucky2017sharp}.
In what follows, we provide the details for the penalty $p_\rho$.

\begin{lemma}\label{lem:cone}
	Suppose that $\rho s<1$ and that the weak decomposability condition of $p_\rho$ with respect to $S:=\operatorname{supp}(\beta^0)$ holds.
	Let $\widehat{\beta}$ be the SRCL estimator defined in \eqref{SRCL_unified}, and let $\Delta:=\widehat{\beta}-\beta^0$.
	Let $c>1$ and suppose that
	$$
	p_\rho^\star\left(\frac{X^\top\epsilon}{\sqrt n\|\epsilon\|_2}\right)\leq\frac{\lambda}{c}.
	$$
	If additionally $c(1-\rho s)>1+\rho s$, then the following bound holds true:
	\begin{equation}\label{important}
	p_\rho^c(\Delta_{S^c})\leq\frac{(c+1)(1-\rho s)}{c(1-\rho s)-(1+\rho s)}p_\rho(\Delta_S),
	\end{equation}
	which indicates $\Delta\in\mathcal{C}_\rho(L,S)$ with
	$$
	L:=\frac{(c+1)(1-\rho s)}{c(1-\rho s)-(1+\rho s)}>1.
	$$
\end{lemma}
\begin{proof}
	By the optimality of $\widehat{\beta}$ in \eqref{SRCL_unified}, we get
	$$
	\frac{1}{\sqrt n}\|y-X\widehat{\beta}\|_2+\lambda p_\rho(\widehat{\beta})\leq\frac{1}{\sqrt n}\|y-X\beta^0\|_2+\lambda p_\rho(\beta^0).
	$$
	Since $y-X\beta^0=\sigma\epsilon$, the convexity of the square-root loss gives
	$$
	\frac{1}{\sqrt n}\|y-X\widehat{\beta}\|_2\geq\frac{1}{\sqrt n}\|y-X\beta^0\|_2-\left\langle	\frac{X^\top\epsilon}{\sqrt n\|\epsilon\|_2},\Delta	\right\rangle,
	$$
	where $-X^\top\epsilon/(\sqrt n\|\epsilon\|_2)$ is a subgradient. Combining the above two inequalities yields
	$$
	\lambda\left\{p_\rho(\widehat{\beta})-p_\rho(\beta^0)\right\}\leq\left\langle\frac{X^\top\epsilon}{\sqrt n\|\epsilon\|_2},\Delta\right\rangle.
	$$
	By the dual norm inequality and the assumed condition, we get
	$$
	\left\langle\frac{X^\top\epsilon}{\sqrt n\|\epsilon\|_2},\Delta\right\rangle\leq p_\rho^\star\left(\frac{X^\top\epsilon}{\sqrt n\|\epsilon\|_2}\right)p_\rho(\Delta)
	\leq\frac{\lambda}{c}p_\rho(\Delta).
	$$
	Therefore, combining the two inequalities above, we obtain
	\begin{equation}\label{pp1c}
	p_\rho(\widehat{\beta})-p_\rho(\beta^0)	\leq\frac{1}{c}p_\rho(\Delta).
	\end{equation}
	Since $\operatorname{supp}(\beta^0)=S$ with $s=|S|$, we have $\beta^0_{S^c}=0$. By the weak decomposability of $p_\rho$ with respect to $S$, we have
	$$
	p_\rho(\widehat{\beta})	=p_\rho(\beta^0+\Delta) 	\geq p_\rho(\beta^0+\Delta_S)+p_\rho^c(\Delta_{S^c}).
	$$
	Moreover, by the reverse triangle inequality $p_\rho(\beta^0+\Delta_S)\geq p_\rho(\beta^0)-p_\rho(\Delta_S)$.
	Hence, it yields that
	\begin{equation}\label{weakno}
	p_\rho(\widehat{\beta})-p_\rho(\beta^0)	\geq	p_\rho^c(\Delta_{S^c})-p_\rho(\Delta_S).
	\end{equation}
	Combining the above inequality \eqref{pp1c}, it gives
	\begin{equation}\label{thisin}
	p_\rho^c(\Delta_{S^c})-p_\rho(\Delta_S)	\leq	\frac{1}{c}p_\rho(\Delta).
	\end{equation}
	For convenience, let $\Delta_{S^c}^{\perp}\in\mathbb{R}^p$ denote the	vector obtained by setting the components of $\Delta$ indexed by $S$
	to zero while keeping the components indexed by $S^c$ unchanged.	Then, $\Delta=\Delta_S+\Delta_{S^c}^{\perp}$. 
	By the triangle inequality, we get that
	\begin{equation}\label{threein}
	p_\rho(\Delta)	\leq	p_\rho(\Delta_S)+p_\rho(\Delta_{S^c}^\perp).
	\end{equation}
	Recall that $\lambda_1=\lambda$ and $\lambda_2=\rho\lambda$, so that	$p(\beta)=\lambda p_\rho(\beta)$. By the definition of $p^c$ in
	\eqref{defpcc}, we have
	\begin{equation}\label{prc1}
	p_\rho^c(\Delta_{S^c})=p^c(\Delta_{S^c})/{\lambda}=(1-\rho s)\|\Delta_{S^c}\|_1+\rho	\sum_{\substack{i<j\\i,j\in S^c}}|\Delta_i-\Delta_j|.
	\end{equation}
	On the other hand, since $\Delta_{S^c}^{\perp}\in\mathbb{R}^p$ is zero when its index located in $S$, we have
	\begin{equation}\label{prc2}
	p_\rho(\Delta_{S^c}^{\perp})=\|\Delta_{S^c}^\perp\|_1+\rho\sum_{i<j}	|(\Delta_{S^c}^{\perp})_i-(\Delta_{S^c}^{\perp})_j|
	=(1+\rho s)\|\Delta_{S^c}\|_1+\rho\sum_{\substack{i<j\\i,j\in S^c}}	|\Delta_i-\Delta_j|.
     \end{equation}
	Consequently, from \eqref{prc1} and \eqref{prc2}, it holds that
	$p_\rho(\Delta_{S^c}^{\perp})=p_\rho^c(\Delta_{S^c})+2\rho s\|\Delta_{S^c}\|_1$.
	Since $\rho s<1$ is assumed, we can get from \eqref{prc1}, that
	$p_\rho^c(\Delta_{S^c})\geq(1-\rho s)\|\Delta_{S^c}\|_1$,
	and hence we get from above both inequalities that 
	\begin{equation}\label{import22}
	p_\rho(\Delta_{S^c}^{\perp})\leq\frac{1+\rho s}{1-\rho s}p_\rho^c(\Delta_{S^c}).
	\end{equation}
	Therefore, combining with the inequality in \eqref{threein}, we get
	$$
	p_\rho(\Delta)	\leq	p_\rho(\Delta_S)	+	\frac{1+\rho s}{1-\rho s}	p_\rho^c(\Delta_{S^c}).
	$$
	Substituting this inequality into the right-hand side of \eqref{thisin}, we have
	$$
	p_\rho^c(\Delta_{S^c})-p_\rho(\Delta_S)	\leq\frac{1}{c}\left(p_\rho(\Delta_S)	+\frac{1+\rho s}{1-\rho s}p_\rho^c(\Delta_{S^c})\right),
	$$
	or, equivalently,
	$$
	\left(1-\frac{1+\rho s}{c(1-\rho s)}\right)p_\rho^c(\Delta_{S^c})
	\leq\left(1+\frac{1}{c}\right)p_\rho(\Delta_S).
	$$
	Moreover, if $c(1-\rho s)>1+\rho s$, then we get 
	$$
	p_\rho^c(\Delta_{S^c})\leq\frac{(c+1)(1-\rho s)}{c(1-\rho s)-(1+\rho s)}p_\rho(\Delta_S),
	$$
	which implies that the desired result \eqref{important} is true. 
    \qed
	
\end{proof}

For convenience, we introduce the following event:
\begin{equation}\label{defbige}
	\mathcal E:=\left\{p_\rho^*\left(\frac{X^\top\epsilon}{\sqrt n\|\epsilon\|_2}\right)\leq\frac{\lambda}{c}\right\}.
\end{equation}
This event encodes a key dual‑norm condition on the noise‑design product $X^\top\epsilon$, which will be invoked in the following analysis.
Building upon the cone membership result from Lemma \ref{lem:cone}, the following theorem derives a non‑asymptotic upper bound on the prediction error for the SRCL estimator. The proof is similar in spirit to the prediction-error argument in Theorem 9 of \cite{stucky2017sharp}, with the details adapted to the penalty $p_\rho$.
 
\begin{theorem}\label{lem:sqrt_oracle}
	Suppose that $\rho s<1$ and that the weak decomposability condition of $p_\rho$ with respect to $S:=\operatorname{supp}(\beta^0)$ holds.
	Let $\widehat{\beta}$ be the SRCL estimator defined in \eqref{SRCL_unified}, and let $\Delta:=\widehat{\beta}-\beta^0$.
	Suppose that the event $\mathcal E$ defined in \eqref{defbige} holds and that $\Delta\in\mathcal C_\rho(L,S)$. Define the noise‑bounded event
	$$
	\mathcal E_0:=\left\{\frac{\|\epsilon\|_2}{\sqrt n}\leq M\right\},
	$$
	where $M>0$ is a constant. Suppose further that, for some fixed constant $r\in(0,1)$, it holds that 
	\begin{equation}\label{rbound}
	\frac{\lambda\sqrt{s}}{\kappa_\rho(L,S)}\leq r.
	\end{equation}
	Here, $\kappa_\rho(L,S)$ is the $p_\rho$-compatibility constant defined in \eqref{kappals}. 
	Then, on the joint event $\mathcal E\cap\mathcal E_0$, it holds that
	$$
	\frac1{\sqrt n}\|X\Delta\|_2\leq C_0\sigma\lambda\frac{\sqrt{s}}{\kappa_\rho(L,S)},
	$$
	where $C_0>0$ is a positive constant depending only on the given constants $M$, $r$, $c$, and $L$.
\end{theorem}
\begin{proof}
	Let	$a:={\|X\Delta\|_2}/{\sqrt n}$ and	$b:={\sigma\|\epsilon\|_2}/{\sqrt n}$.
	By the optimality of $\widehat{\beta}$ and the relation
	$y-X\beta^0=\sigma\epsilon$, we have
	$$
	\frac{1}{\sqrt n}\|\sigma\epsilon-X\Delta\|_2	+\lambda p_\rho(\widehat{\beta})	\leq	b+\lambda p_\rho(\beta^0).
	$$
	Using the fact $p_\rho^c(\Delta_{S^c})\geq 0$ and the inequality in \eqref{weakno}, i.e., $p_\rho(\widehat{\beta})-p_\rho(\beta^0)	\geq p_\rho^c(\Delta_{S^c})-p_\rho(\Delta_S)$,
	we have
	\begin{equation}\label{longin}
	\frac{1}{\sqrt n}\|\sigma\epsilon-X\Delta\|_2	\leq	b+\lambda \Big(p_\rho(\beta^0)- p_\rho(\widehat{\beta})\Big)\leq	b+\lambda \Big(p_\rho(\Delta_S)-p_\rho^c(\Delta_{S^c})\Big)\leq	b+\lambda p_\rho(\Delta_S).
	\end{equation}
     Secondly, since $\rho s<1$ and $\Delta\in\mathcal C_\rho(L,S)$, i.e., $p_\rho^c(\Delta_{S^c})\leq Lp_\rho(\Delta_S)$, combining with \eqref{import22}, we have
	$$
	p_\rho(\Delta_{S^c}^{\perp})
	\leq
	\frac{1+\rho s}{1-\rho s}
	p_\rho^c(\Delta_{S^c})
	\leq
	\frac{1+\rho s}{1-\rho s}
	Lp_\rho(\Delta_S).
	$$
	Therefore, noting $\Delta=\Delta_S+\Delta^\perp_{S^c}$ and using triangle inequality, we get 
	\begin{equation}\label{boundp}
	p_\rho(\Delta)	\leq	p_\rho(\Delta_S)	+p_\rho(\Delta_{S^c}^{\perp})	\leq	\eta p_\rho(\Delta_S),
	\end{equation}
	where $\eta:=1+{L(1+\rho s)}/{(1-\rho s)}$. 
	Thirdly, we define
	$$
	q:=	\frac{\langle X^\top\epsilon,\Delta\rangle}{\sqrt n\|\epsilon\|_2}.
	$$
	On the event $\mathcal E$, using the dual norm inequality, it gives
	$$
	q	=\left\langle\frac{X^\top\epsilon}{\sqrt n\|\epsilon\|_2},\Delta\right\rangle\leq p_\rho^\star\left(\frac{X^\top\epsilon}{\sqrt n\|\epsilon\|_2}	\right)p_\rho(\Delta).
	$$	
	By the definition of $\mathcal E$ and the bound on
	$p_\rho(\Delta)$ given in  \eqref{boundp}, it follows that
	$$
	q\leq\frac{\lambda}{c}p_\rho(\Delta)	\leq	\frac{\lambda \eta}{c}p_\rho(\Delta_S).
	$$
	Fourthly, by the definitions of	$a$, $b$, and $q$, we have
	$$
	\frac{1}{n}\|\sigma\epsilon-X\Delta\|_2^2	=	b^2-2bq+a^2.
	$$
	Hence, from \eqref{longin}, we get 
	$$
	b^2-2bq+a^2	\leq\left(b+\lambda p_\rho(\Delta_S)\right)^2.
	$$
	Expanding the right-hand side and cancelling $b^2$, it gives $	a^2	\leq	2bq	+	2b\lambda p_\rho(\Delta_S)	+	\lambda^2p_\rho^2(\Delta_S)$.
	Using the above bound on $q$, we obtain
	\begin{equation}\label{a2b}
	a^2	\leq	2b\lambda	\left(1+\frac{\eta}{c}\right)	p_\rho(\Delta_S)	+	\lambda^2p_\rho^2(\Delta_S).
	\end{equation}
	Fifthly, since	$\Delta\in\mathcal C_\rho(L,S)$, by the definition of	$\kappa_\rho(L,S)$, whenever $p_\rho(\Delta_S)>0$, it gives that 
	$$
	\frac{\sqrt{s}\|X\Delta\|_2}{\sqrt n\,p_\rho(\Delta_S)}\geq	\kappa_\rho(L,S),
	$$
	which indicates 
	$$
	p_\rho(\Delta_S)\leq\frac{\sqrt{s}\|X\Delta\|_2}{\kappa_\rho(L,S)\sqrt{n}}=\frac{\sqrt{s}}{\kappa_\rho(L,S)}a.
	$$
	If $p_\rho(\Delta_S)=0$, then the preceding inequality for $a^2$
	implies $a=0$, and the desired result follows immediately. Therefore,
	we may assume that $p_\rho(\Delta_S)>0$ and use the above bound.
	Substituting the bound of $p_\rho(\Delta_S)$ into the preceding inequality \eqref{a2b} yields
	\begin{equation}\label{a22bd}
	a^2	\leq	2b\lambda	\left(1+\frac{\eta}{c}\right)	\frac{\sqrt{s}}{\kappa_\rho(L,S)}a	+	\frac{\lambda^2s}{\kappa_\rho^2(L,S)}a^2.
	\end{equation}
	By the assumed condition \eqref{rbound}, we have
	$$
	\frac{\lambda^2s}{\kappa_\rho^2(L,S)}\leq r^2.
	$$
	Therefore, the inequality \eqref{a22bd} can be reorganized as 
	$$
	(1-r^2)a^2\leq  2b\lambda\left(1+\frac{\eta}{c}\right)\frac{\sqrt{s}}{\kappa_\rho(L,S)}a.
	$$
	If $a=0$, the result is immediate. Otherwise, since $r\in(0,1)$,
	we have $1-r^2>0$, and dividing both sides by $a$ gives
	$$
	a	\leq	\frac{2b\lambda}{1-r^2}	\left(1+\frac{\eta}{c}\right)\frac{\sqrt{s}}{\kappa_\rho(L,S)}.
	$$
	Finally, on the event $\mathcal E_0$, we have $b={\sigma\|\epsilon\|_2}/{\sqrt n}	\leq	M\sigma$.
	Thus, on the event $\mathcal E\cap\mathcal E_0$, it holds by noting the definition of $a$ that
	$$
	a=\frac1{\sqrt n}\|X\Delta\|_2	\leq	\frac{2M}{1-r^2}	\left(1+\frac{\eta}{c}\right)	\frac{\sigma\lambda\sqrt{s}}	{\kappa_\rho(L,S)}.
	$$
	Setting
	$$
	C_0:=\frac{2M}{1-r^2}\left(1+\frac{\eta}{c}\right),
	$$
	we obtain
	$$
	\frac1{\sqrt n}\|X\Delta\|_2\leq C_0\sigma\lambda\frac{\sqrt{s}}{\kappa_\rho(L,S)}.
	$$
	This completes the proof.
\end{proof}

Combining the pivotal tuning result, the cone condition in
Lemma \ref{lem:cone}, and the prediction-error bound in
Theorem \ref{lem:sqrt_oracle}, we obtain the following
high-probability prediction-error bound.

\begin{theorem}[Prediction Error Bound]\label{thm:prediction}
	Suppose that Assumptions \ref{assumption_noise}--\ref{assumption_design} hold and that $\rho s<1$.
	Let $\widehat{\beta}$ be the SRCL estimator defined in \eqref{SRCL_unified}, and let $\Delta=\widehat{\beta}-\beta^0$.
	Let $c>1$ satisfy $c(1-\rho s)>1+\rho s$, and define
	\begin{equation*}
		L:=\frac{(c+1)(1-\rho s)}{c(1-\rho s)-(1+\rho s)}>1,
	\end{equation*}
	i.e., the cone parameter obtained in Lemma \ref{lem:cone}.
	Suppose that $\lambda\geq cq_{1-\alpha/2}$, where $q_{1-\alpha/2}$ is the $(1-\alpha/2)$-quantile of $\Lambda$.
	Assume $\kappa_\rho(L,S)>0$ and
	$$
	\frac{\lambda\sqrt{s}}{\kappa_\rho(L,S)}\leq r<1.
	$$
	Then
	$$
	\mathbb{P}\left(\frac1{\sqrt n}\|X(\widehat{\beta}-\beta^0)\|_2	\leq C\sigma\lambda	\frac{\sqrt{s}}{\kappa_\rho(L,S)}	\right)	\geq 1-\alpha,
	$$
	where $C>0$ is a positive constant independent of $n$ and $p$.
\end{theorem}
\begin{proof}
	By the definition of $q_{1-\alpha/2}$ and the choice $\lambda\geq cq_{1-\alpha/2}$, Theorem \ref{thm:pivotal_tuning} yields
	$\mathbb P(\mathcal E)\geq	1-{\alpha}/{2}$.
	Moreover, since Assumption \ref{assumption_noise} shows that $\epsilon_i$ are independent standard Gaussian random variables, then $\|\epsilon\|_2^2\sim\chi_n^2$, i.e., the squared $\ell_2$-norm of $\epsilon$ follows a chi-square distribution.
	 Let
	$$
	M:=	\left(1+2\sqrt{\frac{\log(2/\alpha)}{n}}+\frac{2\log(2/\alpha)}{n}\right)^{1/2},
	$$
	and define an event $\mathcal E_{0}:=\left\{{\|\epsilon\|_2}/{\sqrt n}\leq M\right\}$. 
	By the Laurent--Massart inequality \citep{laurent2000adaptive}, it yields that $\mathbb P(\mathcal E_{0})\geq 1-{\alpha}/{2}$.
	Therefore, by the union bound, we get 
	$$
		\mathbb{P}\left(\mathcal{E}\cap\mathcal{E}_0\right)	=1-\mathbb{P}\left(\mathcal{E}^c\cup\mathcal{E}_0^c\right)
		\geq 1-\mathbb{P}\left(\mathcal{E}^c\right)-\mathbb{P}\left(\mathcal{E}_0^c\right)	\geq 1-\frac{\alpha}{2}-\frac{\alpha}{2}=1-\alpha.
	$$
	On the event $\mathcal E$, by Lemma \ref{lem:cone}, we know that $p_\rho^c(\Delta_{S^c})	\leq Lp_\rho(\Delta_S)$ and  $\Delta\in\mathcal C_\rho(L,S)$.
	Then, combined with the condition 
	$$
	\frac{\lambda\sqrt{s}}{\kappa_\rho(L,S)}\leq r<1,
	$$
	all conditions of Theorem \ref{lem:sqrt_oracle} are satisfied on
	$\mathcal E\cap\mathcal E_{0}$. Therefore, Theorem \ref{lem:sqrt_oracle} gives 
	$$
	\frac1{\sqrt n}\|X\Delta\|_2	\leq	C_0\sigma\lambda	\frac{\sqrt{s}}{\kappa_\rho(L,S)},
	$$
	where
	$$
	C_0=\frac{2M}{1-r^2}\left(1+\frac{\eta}{c}\right),	\qquad\eta=	1+\frac{1+\rho s}{1-\rho s}L.
	$$
	Since $\Delta=\widehat{\beta}-\beta^0$, it follows that
	$$
	\frac1{\sqrt n}\|X(\widehat{\beta}-\beta^0)\|_2	\leq C_0\sigma\lambda\frac{\sqrt{s}}{\kappa_\rho(L,S)}
	$$
	on the event $\mathcal E\cap\mathcal E_{0}$.
	Combining this inequality with $\mathbb P\left(\mathcal E\cap\mathcal E_{0}\right)\geq1-\alpha$,
	we obtain that 
	$$
	\mathbb P\left(\frac1{\sqrt n}\|X(\widehat{\beta}-\beta^0)\|_2	\leq	C_0\sigma\lambda	\frac{\sqrt{s}}{\kappa_\rho(L,S)}\right)\geq1-\alpha.
	$$
	We may thus take the universal constant $C=C_0$ in the statement of Theorem \ref{thm:prediction}.
	The proof is complete.
\end{proof}

\noindent\textit{Zero-signal case.}
If $\beta^0=0$, no compatibility constant is needed. On $\mathcal E$, optimality and convexity of the square-root loss give
\[
\lambda p_\rho(\widehat\beta)
\le \left\langle\frac{X^\top\epsilon}{\sqrt n\|\epsilon\|_2},\widehat\beta\right\rangle
\le \frac{\lambda}{c}p_\rho(\widehat\beta).
\]
Since $c>1$ and $p_\rho$ is a norm, every minimizer satisfies $\widehat\beta=0$ on $\mathcal E$. With $\lambda\ge cq_{1-\alpha}$, this conclusion holds with probability at least $1-\alpha$.
At the end of this section, we emphasize that the clustered Lasso penalty is weakly decomposable with respect to $S$ under $\rho s<1$ (Theorem~\ref{wdtheorem}). This is a key structural condition in the sharp-oracle-inequality framework of \cite{stucky2017sharp}. Instead of deriving sharp oracle inequalities, we establish a non-asymptotic finite-sample prediction-error bound for the proposed SRCL estimator $\widehat{\beta}$ in Theorem \ref{thm:prediction}. Our theoretical analysis is built upon the $(L,S)$-cone condition (Lemma \ref{lem:cone}) induced by weak decomposability together with the penalty-adapted compatibility constant $\kappa_\rho(L,S)$.

\section{An Efficient Semismooth Newton-based Proximal Augmented Lagrangian Algorithm}\label{sec4}
\setcounter{equation}{0}
In this section, we develop an efficient algorithm for solving the square-root clustered Lasso model \eqref{square_root_clustered_lasso_model} through its dual problem, providing a practical implementation of the SRCL method.
Our algorithm is based on a semismooth Newton augmented Lagrangian framework, whose efficiency relies on exploiting the structure of the dual formulation. 
For this purpose, we first reformulate the primal problem by introducing a slack variable $z\in\mathbb{R}^n$ to separate the square-root loss from the clustered Lasso penalty. 
The resulting equivalent primal problem is
\begin{equation}\label{p_problem}
	\min_{\beta\in\mathbb{R}^p,\;z\in\mathbb{R}^n}
	\Big\{\frac{1}{\sqrt{n}}\|z\|_2+p(\beta)\mid	X\beta-y=z	\Big\},
\end{equation}
whose Lagrangian dual problem takes the following equivalent form
\begin{equation}\label{d_problem}
	\min_{w\in\mathbb{R}^n}
	\left\{	f(w):=	p^*(-X^\top w)	+\delta_{B_2^{(1/\sqrt{n})}}(w)	+\langle w,y\rangle	\right\},
\end{equation}
where $p^*$ denotes the Fenchel conjugate of the clustered Lasso penalty $p$, and $\delta_{B_2^{(1/\sqrt{n})}}$ is the indicator function of the Euclidean ball $B_2^{(1/\sqrt{n})}$ of radius $1/\sqrt n$.

\subsection{A Proximal Augmented Lagrangian Method for the Dual Problem}
Following the perturbation framework of \cite[Definition 11.45]{rockafellar1998variational},
we introduce the proper closed  perturbation convex function
$\widetilde f:\mathbb R^n\times\mathbb R^p\times\mathbb R^n\to(-\infty,+\infty]$
defined by
$$
\widetilde{f}(w, \xi_1, \xi_2) := p^*\left(-X^{\top} w + \xi_1\right) + \delta_{B_2^{(1/\sqrt{n})}}\left(w + \xi_2\right) + \langle w, y \rangle,
$$
where $(\xi_1,\xi_2)$ are perturbation variables.
Clearly, $\widetilde{f}(w,0,0)=f(w)$ for all $w\in\mathbb{R}^n$, i.e., $\widetilde{f}$ is an extension of the function $f$ appearing in the dual objective \eqref{d_problem}.			
The Lagrangian function  $L:\mathbb R^n\times\mathbb R^p\times\mathbb R^n   \to(-\infty,+\infty]$ of \eqref{d_problem} associated with a convex perturbation function  $\widetilde{f}$ is defined by
\begin{align}
	L(w;\beta,z)
	&:= \inf_{(\xi_1,\xi_2)}
	\Big\{ \widetilde f(w,\xi_1,\xi_2) -\langle\beta,\xi_1\rangle	-\langle z,\xi_2\rangle \Big\}\nonumber \\[2mm]
	&= 	\inf_{(\xi_1,\xi_2)} 	\Big\{ p^*(-X^\top w+\xi_1) +\delta_{B_2^{(1/\sqrt{n})}}(w+\xi_2)  +\langle w,y\rangle 	-\langle\beta,\xi_1\rangle 	-\langle z,\xi_2\rangle
	\Big\} \nonumber\\[2mm]
	&= \inf_{\xi_1} \Big\{p^*(-X^\top w+\xi_1) -\langle\beta,\xi_1\rangle \Big\} + \inf_{\xi_2}
	\Big\{\delta_{B_2^{(1/\sqrt{n})}}(w+\xi_2)	-\langle z,\xi_2\rangle\Big\} +\langle w,y\rangle,\label{eq43}
\end{align}
where $\beta\in\mathbb R^p$ and $z\in\mathbb R^n$ are multipliers.		
	The first term in \eqref{eq43} is equivalent to 
	\begin{align*}
		\inf_{\xi_1}  \left\{  p^*(-X^\top w+\xi_1)	-\langle\beta,\xi_1\rangle  \right\}
		=    	\inf_u   \left\{ p^*(u) -\langle\beta,u+X^\top w\rangle  \right\}      
		&=    	-\langle\beta,X^\top w\rangle  +\inf_u  \left\{  p^*(u)-\langle\beta,u\rangle \right\}      \\
		&=    	-p(\beta)-\langle\beta,X^\top w\rangle,
	\end{align*}
	whereas the second term is equivalent to
	\begin{align*}
		\inf_{\xi_2} \left\{ \delta_{B_2^{(1/\sqrt{n})}}(w+\xi_2) -\langle z,\xi_2\rangle  \right\}
		=    	\inf_v  \left\{  \delta_{B_2^{(1/\sqrt{n})}}(v)  -\langle z,v-w\rangle    	\right\}      
		&=    	\langle z,w\rangle  +   \inf_v   \left\{  \delta_{B_2^{(1/\sqrt{n})}}(v) -\langle z,v\rangle \right\}      \\
		&=    	\langle z,w\rangle -\delta_{B_2^{(1/\sqrt{n})}}^*(z)      
		=    	\langle z,w\rangle -\frac{1}{\sqrt{n}}\|z\|_2.
	\end{align*}
	Therefore, the associated Lagrangian function   $L(w;\beta,z)$  takes the following form
	\begin{equation}\label{lagfun}
		L(w;\beta,z)= -p(\beta) - \frac{1}{\sqrt{n}}\|z\|_2 + \langle y+z-X\beta,  w \rangle.
	\end{equation}
	The signs in the above Lagrangian function are determined by the sign convention adopted in the perturbation framework.
	
		Following the proximal augmentation framework in \cite[Example 11.57]{rockafellar1998variational} and using a similar derivation with \cite{li2018highly}, for any $\sigma>0$, we obtain the proximal augmented Lagrangian function as follows
	\begin{align*}\label{aug_L_function}
		&{L}_{\sigma}(w;\beta,z) \\
		=& \sup_{\substack{(\xi_1, \xi_2) }} \left\{ L(w; \xi_1, \xi_2) - \frac{1}{2\sigma}\|\xi_1 - \beta\|_2^2 - \frac{1}{2\sigma}\|\xi_2 - z\|_2^2 \right\} \\[2mm]
		=& \sup_{\xi_1}	\left\{	-p(\xi_1)-\langle X^\top w,\xi_1\rangle	-\frac{1}{2\sigma}\|\xi_1-\beta\|_2^2	\right\} 			
		+\sup_{\xi_2}	\left\{	-\frac{1}{\sqrt{n}}\|\xi_2\|_2+\langle w,\xi_2\rangle	-\frac{1}{2\sigma}\|\xi_2-z\|_2^2	\right\}+\langle w,y\rangle\\[2mm]
		=& - \Phi_{\sigma p(\cdot)} \Big(\beta-\sigma X^{\top}w\Big) + \frac{1}{2\sigma}\|\beta-\sigma X^{\top}w\|_2^2 - \frac{1}{2\sigma}\|\beta\|_2^2\\[2mm]
		&
		- \Phi_{\frac{\sigma}{\sqrt{n}} \|\cdot\|_{2}} \Big(z+\sigma w\Big) + \frac{1}{2\sigma}\|z+\sigma w \|_2^2 - \frac{1}{2\sigma}\|z\|_2^2 +\langle w,y\rangle,
	\end{align*}
	where $\Phi$ denotes the Moreau envelope defined in \eqref{moreau}. 
	Moreover, by the definition of the proximal mapping, we know that each of the two supremum subproblems admits a unique maximizer given by the corresponding proximal mapping. Specifically,
	\begin{align*}
		\arg\max_{\xi_1\in \mathbb{R}^p}\left\{	-p(\xi_1)-\langle X^\top w,\xi_1\rangle	-\frac{1}{2\sigma}\|\xi_1-\beta\|_2^2\right\}
		&=	\operatorname{Prox}_{\sigma p(\cdot)}	\big(\beta-\sigma X^\top w\big),\\[2mm]
		\arg\max_{\xi_2\in \mathbb{R}^n}	\left\{	-\frac{1}{\sqrt n}\|\xi_2\|_2+\langle w,\xi_2\rangle-\frac{1}{2\sigma}\|\xi_2-z\|_2^2\right\}
		&=	\operatorname{Prox}_{\frac{\sigma}{\sqrt{n}}\|\cdot\|_2}	\big(z+\sigma w\big).
	\end{align*}
	Therefore, the multipliers $(\beta,z)$ in the proximal augmented Lagrangian method can be  explicitly updated.

    Based on the proximal augmented Lagrangian function, we propose a proximal augmented Lagrangian method for solving the dual problem \eqref{d_problem}; see Algorithm~\ref{Algorithm 4.1}. The main computational effort lies in Step~1, which involves solving the subproblem \eqref{sub_problem}. We employ an efficient semismooth Newton method to solve this subproblem, which will be presented in the next subsection. We refer to the resulting method as the semismooth Newton-based proximal augmented Lagrangian method, abbreviated as N-ALM.

	\begin{framed}
		\noindent
		\textbf{Algorithm 5.1: N-ALM}\label{Algorithm 4.1}
		\vskip 3.0mm \hrule \vskip 3mm
		\noindent{\textbf{Initialization.}} Given are a positive nondecreasing sequence $\{\sigma_k\}_{k\geq 0}$ and a positive nonincreasing sequence $\{\tau_k\}_{k\geq 0}$. Choose $(w^0, \beta^0, z^0) \in \mathbb{R}^n \times \mathbb{R}^p \times \mathbb{R}^n$. Here $\beta^0$ denotes an algorithmic initial point, distinct from the true regression coefficient used in the statistical analysis; $\sigma_k$ is an algorithmic penalty parameter, distinct from the noise level $\sigma$. Set $k=0$. Repeat the following steps until a proper stopping criterion is satisfied by the sequence $\{(w^k, \beta^k, z^k)\}$:\\
		\noindent{\textbf{Step 1.}} Approximately solve the following subproblem via Sub-Algorithm SSN listed later:
		\begin{equation}\label{sub_problem}
			w^{k+1} \approx \underset{w \in \mathbb{R}^n}{\operatorname{argmin}} \left\{ \psi_k(w) := {{L}}_{\sigma_k}(w; \beta^k, z^k) + \frac{\tau_k}{2\sigma_k}\|w - w^k\|_2^2 \right\}.
		\end{equation}
		\noindent{\textbf{Step 2.}} Update the multipliers via the following rule:
		\begin{equation}\label{eq:nalm-multiplier-updates}
		(\beta^{k+1}, z^{k+1}) = \left( \operatorname{Prox}_{\sigma_k p(\cdot)}\Big(\beta^k-\sigma_k X^{\top}w^{k+1}\Big), \quad \operatorname{Prox}_{\frac{\sigma_k}{\sqrt{n}} \|\cdot\|_{2}}\Big(z^k+\sigma_k w^{k+1}\Big) \right).
		\end{equation}
	\end{framed}	
	
	Before proceeding to the convergence analysis, we specify the approximation accuracy required in \textbf{Step 1} of Algorithm~\ref{Algorithm 4.1}. For this purpose, define the positive definite block diagonal matrix
	\begin{equation}\label{block diagonal matrix}
		\mathcal{M}_k := \operatorname{Diag}(\tau_k I_n, I_p, I_n) \succ 0.
	\end{equation}
	Following \cite{li2020asymptotically}, we employ the following implementable stopping criteria for the subproblem \eqref{sub_problem}, which provide a practical realization of the inexact augmented Lagrangian framework developed in \cite{rockafellar1976augmented}:
	\begin{align}
		\|\nabla \psi_k(w^{k+1})\|_2
		&\le \frac{\min(\sqrt{\tau_k},1)}{\sigma_k}\varepsilon_k,\tag{A}\label{eq:nalm-stop-a}\\
		\|\nabla \psi_k(w^{k+1})\|_2
		&\le
		\frac{\delta_k\min(\sqrt{\tau_k},1)}{\sigma_k}
		\|(w^{k+1},\beta^{k+1},z^{k+1})
		-(w^k,\beta^k,z^k)\|_{\mathcal{M}_k},\tag{B}\label{eq:nalm-stop-b}
	\end{align}
	where $\{\varepsilon_k\}_{k\ge0}$ and $\{\delta_k\}_{k\ge0}$ are positive summable sequences satisfying $\delta_k<1$ for all $k$.
	
	\subsection{A semismooth Newton Method for Subproblem \eqref{sub_problem}}
	Noting that $\psi_k(\cdot)$ in \eqref{sub_problem} is convex and continuously differentiable, it is easy to show that its gradient takes the following compact form
	\begin{align}\label{sub_equ_problem}
		\nabla \psi_k(w) &= - X\operatorname{Prox}_{\sigma_k p(\cdot)}\big(\beta^k-\sigma_k X^{\top}w\big) + \operatorname{Prox}_{\frac{\sigma_k}{\sqrt{n}} \|\cdot\|_{2}}\big(z^k+\sigma_k w\big) + y + \frac{\tau_k}{\sigma_k}\big(w - w^k\big) \nonumber\\[2mm]
		& = - \sigma_kX\operatorname{Prox}_{p(\cdot)}\big(\beta^k/\sigma_k- X^{\top}w\big) + \sigma_k\operatorname{Prox}_{\frac{1}{\sqrt{n}}\|\cdot\|_{2}}\big(z^k/\sigma_k+ w\big) + y + \frac{\tau_k}{\sigma_k}\big(w - w^k\big).
	\end{align}
	The second equality follows from the fact that  $p(\cdot)$ and $\|\cdot\|_2$ are positively homogeneous convex functions of degree one. Consequently, $\operatorname{Prox}_{\sigma_k p(\cdot)}(x) = \sigma_k \operatorname{Prox}_{p(\cdot)}(x/\sigma_k)$ and 
	$\operatorname{Prox}_{\frac{\sigma_k}{\sqrt{n}}\|\cdot\|_{2}}(x) = \sigma_k \operatorname{Prox}_{\frac{1}{\sqrt{n}}\|\cdot\|_{2}}(x/\sigma_k)$ for inputs in $\mathbb R^p$ and $\mathbb R^n$, respectively. 
	Therefore, solving \eqref{sub_problem} is equivalent to finding a solution of the following nonlinear system of equations
	\begin{equation}\label{nabla}
		\nabla \psi_k(w)=0,
	\end{equation}
   which is  semismooth due to the semismoothness of the proximal mappings associated with $p(\cdot)$ and the $\ell_2$-norm $\|\cdot\|_2$, as established in \cite[Remark 2.12]{lin2019efficient} and \cite[Lemma 1]{yuan2018efficient}, respectively. Therefore, it is reasonable to employ a semismooth Newton method to solve \eqref{nabla}, or equivalently, \eqref{sub_problem}.

	Consider the following set-valued mapping $\hat{\partial}^2\psi_k: \mathbb{R}^n \rightrightarrows \mathbb{R}^{n\times n}$:
	$$
	\hat{\partial}^2\psi_k(w)
	:=
	\left\{\sigma_k XUX^\top+\sigma_k V+\frac{\tau_k}{\sigma_k}I_n\right\},
	$$
	where
	$$
	U\in\partial\operatorname{Prox}_{p(\cdot)}
	\left(\frac{\beta^k}{\sigma_k}-X^\top w\right),
	\qquad
	V\in\partial\operatorname{Prox}_{\frac{1}{\sqrt n}\|\cdot\|_2}
	\left(\frac{z^k}{\sigma_k}+w\right),
	$$
	and $w\in\mathbb{R}^n$. Here, $\partial\operatorname{Prox}_{p(\cdot)}(\cdot)$ and $\partial\operatorname{Prox}_{\frac{1}{\sqrt n}\|\cdot\|_2}(\cdot)$ denote the Clarke generalized Jacobians of the corresponding proximal mappings. The former is distinguished from the computable generalized derivative mapping \eqref{M}; the latter is given in \eqref{l2clarke}.
	It follows from \cite{hiriart1984generalized} that
	$$
	\partial^2\psi_k(w)\subseteq\hat{\partial}^2\psi_k(w),
	$$
	where $\partial^2\psi_k(w):=\partial(\nabla\psi_k)(w)$ denotes the Clarke generalized Hessian of $\psi_k$.
	The Clarke generalized Jacobians of the proximal mappings are nonempty, compact, and upper semicontinuous; the same properties therefore hold for $\hat{\partial}^2\psi_k$. The inclusion above is an outer approximation and need not be an equality. Moreover, since $U$ and $V$ are positive semidefinite as reviewed in Section \ref{preli} and that $\tau_k/\sigma_k>0$, every element $H\in\hat{\partial}^2\psi_k(w)$ is actually symmetric positive definite.
	For computation, we use the generalized Newton derivative mapping
	\[
	\mathcal H_k(w):=\left\{\sigma_k XUX^\top+\sigma_k V+\frac{\tau_k}{\sigma_k}I_n:
	\ U\in\mathcal M(\beta^k/\sigma_k-X^\top w),\quad
	V\in\partial\operatorname{Prox}_{\|\cdot\|_2/\sqrt n}(z^k/\sigma_k+w)\right\}.
	\]
	The semismoothness properties in \eqref{prox_semismooth} and \eqref{l2clarke} imply that $\nabla\psi_k$ is strongly semismooth with respect to $\mathcal H_k$. This mapping is nonempty, compact, and upper semicontinuous, and all its elements are symmetric positive definite. It need not coincide with the Clarke generalized Hessian.
	When applying the semismooth Newton method to solve \eqref{nabla}, starting from $w^{k,0}$, we generate a sequence $\{w^{k,i}\}$ according to
	$$
	w^{k,i+1}:=w^{k,i}+\alpha_i\Delta w^{i},
	\qquad
	\text{where}\qquad
	H^{i}\Delta w^{i}+\nabla\psi_k(w^{k,i})=0,
	$$
	where $H^i\in\mathcal H_k(w^{k,i})$, $\alpha_i$ is a step length, and $\Delta w^i$ is the Newton direction.
	
	In light of the above analysis, we present the framework of the semismooth Newton method (SSN) for solving the subproblem \eqref{sub_problem} as follows.

	\begin{framed}
		\noindent
		\textbf{Sub-Algorithm 5.2: SSN} \label{Algorithm 4.2}
		\vskip 3.0mm \hrule \vskip 3mm
		\noindent{\textbf{Initialization.}} Given $\sigma_k>0$ and $\tau_k>0$. Choose $\mu\in (0,1/2)$, $\bar{\eta}\in(0,1)$, $\bar{\tau}\in(0,1]$, $\delta\in(0,1)$, and $w^{k,0}\in\mathbb{R}^n$. Let $i:=0$. Do the following steps iteratively:\\
		\noindent{\textbf{Step 1.}} Select $U^{i}\in\mathcal M(\beta^k/\sigma_k-X^{\top}w^{k,i})$ and $V^{i}\in\partial\text{Prox}_{\frac{1}{\sqrt{n}} \|\cdot\|_{2}}(z^k/\sigma_k+ w^{k,i})$, and set ${H}^{i}:=\sigma_k XU^{i}X^{\top}+\sigma_k V^{i}+\tau_k \sigma_k^{-1} I_{n}$. Solve the linear system \eqref{linear equation} directly or by the conjugate gradient (CG) method to obtain $\Delta w^i$ satisfying  
		\begin{equation}\label{linear equation}
		{H}^{i}\Delta w^{i}=-\nabla \psi_k(w^{k,i})
		\end{equation}
		such that
		$$
		\| {H}^{i}\Delta w^{i}+\nabla \psi_k(w^{k,i})\|_2\leq \min \{\bar{\eta}, \|\nabla \psi_k(w^{k,i})\|_2^{1+\bar{\tau}}\}.
		$$
		\noindent{\textbf{Step 2.}} Set $\alpha_i:=\delta^{m_i}$, where $m_i$ is the first nonnegative integer $m$ such that
		$$
		\psi_k (w^{k,i}+\delta^{m}\Delta w^{i})\leq \psi_k (w^{k,i})+\mu \delta^{m}\big\langle \nabla \psi_k (w^{k,i}), \Delta w^{i}\big\rangle.
		$$
		\noindent{\textbf{Step 3.}} Set $w^{k,i+1}:=w^{k,i}+\alpha_i\Delta w^{i}$. Let $i:=i+1$ and go to \textbf{Step 1}.
	\end{framed}
	
	We see that the main computational cost of Sub-Algorithm SSN lies in solving the Newton linear system \eqref{linear equation}, which takes the form
	\begin{equation}\label{linear equation_equ form}
		\left(
		\sigma_k XU^iX^\top
		+\sigma_k V^i
		+\frac{\tau_k}{\sigma_k}I_n
		\right)\Delta w^i
		=
		-\nabla\psi_k(w^{k,i}),
	\end{equation}
	where
	$$
	U^i\in\mathcal M
	\left(\frac{\beta^k}{\sigma_k}-X^\top w^{k,i}\right),
	\qquad
	V^i\in\partial\operatorname{Prox}_{\frac{1}{\sqrt n}\|\cdot\|_2}
	\left(\frac{z^k}{\sigma_k}+w^{k,i}\right).
	$$
	To solve \eqref{linear equation_equ form} efficiently, we exploit the special structures of the matrices $U^i$ and $V^i$. As shown in \cite{lin2019efficient}, the matrix $XU^iX^\top$ admits the decomposition
	\begin{equation}\label{XUX_structure}
		XU^iX^\top	=X_\gamma X_\gamma^\top	+X_\alpha \widetilde U^i (\widetilde U^i)^\top X_\alpha^\top,
	\end{equation}
	where $X_\alpha$ and $X_\gamma$ are the submatrices of $X$ corresponding to the index sets $\alpha$ and $\gamma$, respectively, and $\tilde U^i$ is defined as in \cite{lin2019efficient}.
	
	We conclude this subsection by stating the convergence result of SSN, which is well established in the literature; see, e.g., \citep[Proposition~3.3 and Theorem~3.4]{zhao2010newton}.
	\begin{theorem}\label{globc}
		Let $\{w^{k,i}\}_{i\geq 0}$ be the sequence generated by Sub-Algorithm SSN. Then, $\{w^{k,i}\}$ converges to the unique optimal solution ${w}^{k+1}$ of Subproblem SSN. Moreover, it holds that
		$$
		\|w^{k,i+1}-{w}^{k+1}\|_2=	O\left(\|w^{k,i}-{w}^{k+1}\|_2^{1+\bar{\tau}}\right),
		$$
		for all sufficiently large $i$, where $\bar{\tau}\in(0,1]$ is the constant used in Sub-Algorithm SSN.
	\end{theorem}
	
\subsection{Convergence Analysis of N-ALM}\label{convana}

In this section, we analyze the local convergence rate of Algorithm~\ref{Algorithm 4.1} using the proximal point framework of \cite{li2020asymptotically}. Their convergence theory provides a general approach to the rate analysis, but its application requires an error bound for the associated KKT mapping. Establishing this bound for the square-root clustered Lasso model is the main task of this section. Although the penalty $p$ is polyhedral convex, the subdifferential of the Euclidean norm is generally not a polyhedral multifunction. Thus, the error bound for the full KKT system does not follow directly from the polyhedrality of $p$.

The results of \cite{zhou2017unified} provide useful error-bound properties for the individual subdifferential mappings. However, these properties alone do not establish an error bound for the KKT system, where the primal and dual variables are coupled through $X$. We therefore first establish a local error bound for the primal subdifferential and then use the KKT conditions to obtain an error bound for the full system under the conditions stated below. This supplies the model-specific condition needed to apply the local rate analysis of \cite{li2020asymptotically}.

Let $F$ denote the objective function in \eqref{square_root_clustered_lasso_model}, and write
\begin{equation}\label{eq:primal-objective-solutions}
	F(\beta):=\frac1{\sqrt n}\|X\beta-y\|_2+p(\beta),\qquad
	F^*:=\min_\beta F(\beta),\qquad
	S_{\mathrm P}:=\operatorname*{arg\,min}_\beta F(\beta).
\end{equation}
Since $F$ is continuous and convex and $F(\beta)\ge\lambda_1\|\beta\|_1$ with $\lambda_1>0$, $S_{\mathrm P}$ is nonempty, compact, and convex. To relate primal optimality to the full KKT system, write $u:=(w,\beta,z)$ and define its three component mappings by
\begin{equation}\label{eq:kkt-operator-explicit}
	\mathcal T_X(u):=\{z-X\beta+y\},\qquad
	\mathcal T_p(u):=X^\top w+\partial p(\beta),\qquad
	\mathcal T_2(u):=\tfrac1{\sqrt n}\partial\|z\|_2-w.
\end{equation}
The KKT mapping and its solution set are
\begin{equation}\label{eq:kkt-mapping}
	\mathcal T(u):=\mathcal T_X(u)\times\mathcal T_p(u)\times\mathcal T_2(u),
	\qquad S_{\mathrm{KKT}}:=\mathcal T^{-1}(0).
\end{equation}
By convexity and the subdifferential sum rule, the primal solution set satisfies
\begin{equation}\label{eq:primal-kkt-equivalence}
	S_{\mathrm P}=\{\beta:\text{there exists }w\in\mathbb R^n\text{ such that }(w,\beta,X\beta-y)\in S_{\mathrm{KKT}}\}.
\end{equation}
In particular, $S_{\mathrm{KKT}}$ is nonempty. As the saddle-point operator associated with the closed convex--concave function $L$ in \eqref{lagfun}, $\mathcal T$ is maximal monotone; see \cite{rockafellar1976augmented}.

Suppose that $0<\sigma_k\uparrow\sigma_\infty\le+\infty$ and $\tau_k\downarrow\tau_\infty>0$. Then the matrices in \eqref{block diagonal matrix} satisfy
\begin{equation}\label{eq:palm-metric-bounds}
	\underline m I\preceq\mathcal M_{k+1}\preceq\mathcal M_k\preceq\gamma_0 I,
	\qquad \underline m:=\min\{1,\tau_\infty\}>0,\quad\gamma_0:=\max\{1,\tau_0\}.
\end{equation}
To apply the proximal point convergence results, write $u^k:=(w^k,\beta^k,z^k)$ and let $\mathcal P_k:=(\mathcal M_k+\sigma_k\mathcal T)^{-1}\mathcal M_k$ and $\widetilde u^{k+1}:=\mathcal P_k(u^k)$ denote the exact proximal mapping and update. The multiplier updates \eqref{eq:nalm-multiplier-updates} and criterion \eqref{eq:nalm-stop-a} give
\begin{equation}\label{eq:palm-absolute-proximal-error}
	E_k:=\|u^{k+1}-\widetilde u^{k+1}\|_{\mathcal M_k}
	\le\frac{\sigma_k}{\sqrt{\tau_k}}\|\nabla\psi_k(w^{k+1})\|_2
	\le\varepsilon_k,
\end{equation}
by the argument of \cite[Proposition~3.2]{li2020asymptotically}. Thus, \eqref{eq:palm-metric-bounds}, $\sigma_k\ge\sigma_0>0$, and $\sum_k\varepsilon_k<\infty$ allow us to apply \cite[Theorem~2.3]{li2020asymptotically} with metric-change sequence $\nu_k=0$. The sequence $\{u^k\}$ converges to some $\bar u=(\bar w,\bar\beta,\bar z)\in S_{\mathrm{KKT}}$, where $(\bar\beta,\bar z)$ and $\bar w$ solve \eqref{p_problem} and \eqref{d_problem}, respectively.

To obtain the local rate, we establish an error bound for $\mathcal T$ at $\bar u$ and apply the rate argument of \cite[Theorem~2.5]{li2020asymptotically}. We first derive an error bound for $\partial F$, then use the KKT conditions to extend it to $\mathcal T$ under the conditions stated in Theorem~\ref{thm:kkt-eb}.

\begin{lemma}
	\label{lem:primal-eb}
	For every $\bar\beta\in S_{\mathrm P}$, there exist $r>0$ and $\kappa_{\mathrm P}>0$ such that
	\begin{equation}\label{eq:primal-subdiff-eb}
		\operatorname{dist}(\beta,S_{\mathrm P})\le\kappa_{\mathrm P}\operatorname{dist}\bigl(0,\partial F(\beta)\bigr)
		\qquad\text{whenever }\|\beta-\bar\beta\|_2\le r.
	\end{equation}
\end{lemma}

\begin{proof}
	Fix $\bar\beta\in S_{\mathrm P}$ and set $\bar z:=X\bar\beta-y$. By \eqref{eq:primal-kkt-equivalence}, there exists $\bar w$ such that
	\[
	-X^\top\bar w\in\partial p(\bar\beta),\qquad
	\bar w\in\tfrac1{\sqrt n}\partial\|\bar z\|_2\subseteq B_2^{(1/\sqrt n)}.
	\]
	Since $p$ and $\|\cdot\|_2$ are convex and positively homogeneous, these optimality conditions imply
	\[
	p(\beta)\ge-\langle X^\top\bar w,\beta\rangle,
	\qquad \tfrac1{\sqrt n}\|z\|_2\ge\langle\bar w,z\rangle
	\]
	for all $\beta,z$, with equality at $\bar\beta,\bar z$, respectively. Taking $z=X\beta-y$ shows that $F(\beta)\ge-\langle\bar w,y\rangle$ for every $\beta$, with equality at $\bar\beta$. Hence $F^*=-\langle\bar w,y\rangle$ and
	\begin{equation}\label{eq:primal-gap-decomposition}
		F(\beta)-F^*=
		\bigl[p(\beta)+\langle X^\top\bar w,\beta\rangle\bigr]
		+\bigl[\tfrac1{\sqrt n}\|X\beta-y\|_2-\langle\bar w,X\beta-y\rangle\bigr].
	\end{equation}
	
	Define the affine set $\mathcal A$ and the set $\mathcal C$ by
	\[
	\mathcal A:=\{(\beta,z)\in\mathbb R^p\times\mathbb R^n:z=X\beta-y\},
	\qquad \mathcal C:=\mathcal C_p\times\mathcal C_2,
	\]
	where
	\[
	\mathcal C_p:=\{\beta:-X^\top\bar w\in\partial p(\beta)\},
	\qquad \mathcal C_2:=\{z:\bar w\in\tfrac1{\sqrt n}\partial\|z\|_2\}.
	\]
	By the definitions of $\mathcal C_p$ and $\mathcal C_2$, together with the nonnegativity of the two terms in \eqref{eq:primal-gap-decomposition}, $F(\beta)=F^*$ if and only if $\beta\in\mathcal C_p$ and $X\beta-y\in\mathcal C_2$. Therefore,
	\begin{equation}
		S_{\mathrm P} = \{\beta\in\mathcal C_p:X\beta-y\in\mathcal C_2\} 
		\qquad\text{and}\qquad 
		\mathcal A\cap\mathcal C 	= \{(\beta,X\beta-y):\beta\in S_{\mathrm P}\}.
		\label{eq:primal-solution-representation}
	\end{equation}
	By \eqref{eq:primal-solution-representation},
	\[
	\begin{aligned}
		\operatorname{dist}\bigl((\beta,X\beta-y),\mathcal A\cap\mathcal C\bigr)^2
		&=\inf_{\beta'\in S_{\mathrm P}}\bigl\{\|\beta-\beta'\|_2^2+\|X(\beta-\beta')\|_2^2\bigr\}\\
		&\ge\inf_{\beta'\in S_{\mathrm P}}\|\beta-\beta'\|_2^2
		=\operatorname{dist}(\beta,S_{\mathrm P})^2.
	\end{aligned}
	\]
	To bound the distance to $\mathcal A\cap\mathcal C$, we first estimate the distances to $\mathcal C_2$ and $\mathcal C_p$.

	Since $p^*$ is polyhedral convex, $(\partial p)^{-1}=\partial p^*$ is a polyhedral multifunction by \cite[Theorem~23.10]{Rock35}. Robinson's result on the upper Lipschitz continuity of polyhedral multifunctions \cite[Proposition~1]{robinson1981some} implies that $(\partial p)^{-1}$ is calm at $-X^\top\bar w$ for $\bar\beta$. By the inverse-calmness characterization in \cite[Theorem~3H.3]{dontchev2009implicit}, $\partial p$ is metrically subregular at $\bar\beta$ for $-X^\top\bar w$.  
	As $p$ is proper, lower semicontinuous, and convex, we can apply    \cite[Theorem~3.3]{aragon2008characterization}  to $\partial p $ at $(\bar\beta,-X^\top\bar w)$,  and  obtain a constant $\kappa_p>0$ and $r_p>0$ such that 
	\[
	p(\beta)\ge
	p(\bar\beta)-\langle X^\top\bar w,\beta-\bar\beta\rangle
	+ \frac{1}{\kappa_p} \operatorname{dist}(\beta,\mathcal C_p)^2,  \qquad  \|\beta-\bar \beta\|_2\le r_p. 
	\]
	Using $p(\bar\beta)=-\langle X^\top\bar w,\bar\beta\rangle$,  we obtain  
	\begin{equation}
		\operatorname{dist}(\beta,\mathcal C_p)^2
		\le  \kappa_p  \bigl[p(\beta)+\langle X^\top\bar w,\beta\rangle\bigr],
		\qquad  \|\beta-\bar \beta\|_2\le r_p.
		\label{eq:penalty-gap-growth}
	\end{equation}

	We next establish a local bound on the distance to $\mathcal C_2$. Applying \cite[Proposition~7]{zhou2017unified} to a single group with weight $1/\sqrt n$ and fixed vector $-\bar w$, we obtain
	\begin{equation}
		\mathcal C_2=
		\begin{cases}
			\{0\}, & \|\bar w\|_2<1/\sqrt n,\\
			\{t\sqrt n\,\bar w:t\ge0\}, & \|\bar w\|_2=1/\sqrt n,
		\end{cases}
		\label{eq:loss-zero-set}
	\end{equation}
	Thus $\mathcal C_2$ is polyhedral convex. By \cite[Proposition~8]{zhou2017unified}, $\partial\|\cdot\|_2$ is metrically subregular at $\bar z$ for $\sqrt n\,\bar w$. Positive scaling therefore implies that $\tfrac1{\sqrt n}\partial\|\cdot\|_2$ is metrically subregular at $\bar z$ for $\bar w$. 
	Since $\|\cdot\|_2/\sqrt n$ is proper, lower semicontinuous, and convex, \cite[Theorem~3.3]{aragon2008characterization} gives $\kappa_2>0$ and $r_2>0$
	such that
	\[
	\operatorname{dist}(z,\mathcal C_2)^2
	\le \kappa_2\bigl[\tfrac1{\sqrt n}\|z\|_2-\langle\bar w,z\rangle\bigr]
	\qquad\text{for $\| z-   \bar z\|_2\leq r_2$ },
	\]
	where we used $\|\bar z\|_2/\sqrt n=\langle\bar w,\bar z\rangle$.  
	Let $r :=\min\left\{r_p,\, r_2/\max\{1,\|X\|_2\}\right\}$,   we have $\|(X\beta-y)-(X\bar\beta-y)\|_2\leq r_2$  when $ \beta\in U:=\{\beta\in\mathbb R^p:\|\beta-\bar\beta\|_2\le r\}$ .
	Hence
	\begin{equation}
		\operatorname{dist}(X\beta-y,\mathcal C_2)^2
		\le
		\kappa_2\bigl[\tfrac1{\sqrt n}\|X\beta-y\|_2-\langle\bar w,X\beta-y\rangle\bigr], \qquad \beta\in U.
		\label{eq:loss-gap-growth}
	\end{equation}

	The affine set $\mathcal A$ and the product $\mathcal C=\mathcal C_p\times\mathcal C_2$ are both polyhedral convex sets, and their intersection contains $(\bar\beta,\bar z)$. Thus, \cite[Fact~3]{zhou2017unified} implies that $\{\mathcal A,\mathcal C\}$ is boundedly linearly regular. Since $(\beta,X\beta-y)\in\mathcal A$ and $U$ is bounded, there exists $H>0$ such that, for every $\beta\in U$,
	\[
	\begin{aligned}
		\operatorname{dist}(\beta,S_{\mathrm P})^2
		&\le
		\operatorname{dist}\bigl((\beta,X\beta-y),
		\mathcal A\cap\mathcal C\bigr)^2\\
		&\le
		H^2\operatorname{dist}\bigl((\beta,X\beta-y),
		\mathcal C\bigr)^2\\
		&=
		H^2\left(
		\operatorname{dist}(\beta,\mathcal C_p)^2+
		\operatorname{dist}(X\beta-y,\mathcal C_2)^2
		\right).
	\end{aligned}
	\]
	By \eqref{eq:loss-gap-growth}, \eqref{eq:penalty-gap-growth}, and \eqref{eq:primal-gap-decomposition},
	\[
	\begin{aligned}
		\operatorname{dist}(\beta,S_{\mathrm P})^2
		&\le H^2\Bigl(\kappa_p\bigl[p(\beta)+\langle X^\top\bar w,\beta\rangle\bigr]\\
		&\hspace{3em}+\kappa_2\bigl[\tfrac1{\sqrt n}\|X\beta-y\|_2-\langle\bar w,X\beta-y\rangle\bigr]\Bigr)\\
		&\le H^2\max\{\kappa_p,\kappa_2\}\bigl[F(\beta)-F^*\bigr].
	\end{aligned}
	\]
	Thus, setting $\kappa_P:=H^2\max\{\kappa_p,\kappa_2\}>0$, we have
	\begin{equation}\label{eq:primal-qg}
		F(\beta)-F^*\ge \kappa_P^{-1}\,\operatorname{dist}(\beta,S_{\mathrm P})^2
		\qquad(\beta\in U).
	\end{equation}
	
	Now fix $\beta\in U$ and let $\beta^\pi$ be its Euclidean projection onto the nonempty closed convex set $S_{\mathrm P}$. Since $F$ is convex and $F(\beta^\pi)=F^*$, for every $v\in\partial F(\beta)$ we have
	\[
	F(\beta)-F^*
	\le
	\langle v,\beta-\beta^\pi\rangle
	\le
	\|v\|_2\operatorname{dist}(\beta,S_{\mathrm P}).
	\]
	Together with \eqref{eq:primal-qg}, this implies
	\[
	\operatorname{dist}(\beta,S_{\mathrm P})
	\le \kappa_P\|v\|_2
	\]
	whenever $\beta\notin S_{\mathrm P}$.
	For $\beta\in S_{\mathrm P}$, the inequality also holds since $\operatorname{dist}(\beta,S_{\mathrm P})=0$. Taking the infimum over $v\in\partial F(\beta)$ therefore gives \eqref{eq:primal-subdiff-eb}.
\end{proof}

Lemma~\ref{lem:primal-eb} gives a local error bound for $\partial F$. To establish the local convergence rate of N-ALM, we need an error bound for the full KKT mapping $\mathcal T$. In the following theorem, we combine the primal error bound with the three KKT conditions to bound the distance to $S_{\mathrm{KKT}}$ in terms of $\operatorname{dist}(0,\mathcal T(u))$.
\begin{theorem}
	\label{thm:kkt-eb}
	Let $\bar u:=(\bar w,\bar\beta,\bar z)\in S_{\mathrm{KKT}}$.
	Suppose that either
	\[
	\bar z\ne0,
	\qquad\text{or}\qquad
	\bar z=0
	\quad\text{and}\quad
	\|\bar w\|_2<\frac{1}{\sqrt n}.
	\]
	Then there exist constants $\varepsilon>0$ and $\kappa>0$
	such that
	\begin{equation}
		\operatorname{dist}(u,S_{\mathrm{KKT}})
		\le
		\kappa\,\operatorname{dist}\bigl(0,\mathcal T(u)\bigr)
		\qquad
		\text{for }\|u-\bar u\|_2\le\varepsilon,
		\label{eq:kkt-local-eb}
	\end{equation}
	where $u=(w,\beta,z)$.
\end{theorem}

\begin{proof}
	For $u=(w,\beta,z)$, define
	\begin{equation}\label{eq:kkt-component-residuals}
		r_i(u):=\operatorname{dist}(0,\mathcal T_i(u)),\quad i\in\{X,p,2\},
		\qquad \rho(u):=\operatorname{dist}(0,\mathcal T(u))=\sqrt{r_X(u)^2+r_p(u)^2+r_2(u)^2}.
	\end{equation}
	
	\medskip\noindent\textbf{\textit{Case 1}: $\bar z\ne0$.}
	The function
	\[
	g(v):=\frac{v}{\sqrt n\,\|v\|_2}
	\]
	is Lipschitz continuous on $V:=\{v:\|v-\bar z\|_2\le\|\bar z\|_2/2\}$ with Lipschitz constant $L_g>0$. Let $r_{\mathrm P}>0$ and $\kappa_{\mathrm P}>0$ be as in Lemma~\ref{lem:primal-eb}, and set
	\[
	\varepsilon_1:=\min\left\{r_{\mathrm P},\frac{\|\bar z\|_2}{2\max\{1,\|X\|_2\}}\right\}>0.
	\]
	For $\|u-\bar u\|_2\le\varepsilon_1$, we have
	\[
	\|z-\bar z\|_2\le\varepsilon_1\le\frac{\|\bar z\|_2}{2},\qquad
	\|(X\beta-y)-\bar z\|_2\le\|X\|_2\varepsilon_1\le\frac{\|\bar z\|_2}{2},
	\]
	so $z,X\beta-y\in V$. Hence
	\[
	\partial F(\beta)=\partial p(\beta)+X^\top g(X\beta-y),
	\qquad r_2(u)=\|g(z)-w\|_2.
	\]
	It follows that
	\begin{align}
		\operatorname{dist}(0,\partial F(\beta))
		&\le r_p(u)+\|X\|_2\|g(X\beta-y)-w\|_2\nonumber\\
		&\le r_p(u)+\|X\|_2\bigl(L_g r_X(u)+r_2(u)\bigr)\nonumber\\
		&\le C\rho(u),
		\label{eq:kkt-primal-residual-bound}
	\end{align}
	where $ C =\sqrt{1+(1+L_g^2)\|X\|_2^2}$.
	Let $\beta^\pi$ be the Euclidean projection of $\beta$ onto $S_{\mathrm P}$. Since $\|\beta-\bar\beta\|_2\le\varepsilon_1\le r_{\mathrm P}$, Lemma~\ref{lem:primal-eb} gives
	\begin{equation}\label{eq:kkt-set-primal-distance}
		\|\beta-\beta^\pi\|_2\le \kappa_{\mathrm P}C\rho(u).
	\end{equation}
	Since $S_{\mathrm P}$ is closed and convex and $\bar\beta\in S_{\mathrm P}$,
	\[
	\|\beta^\pi-\bar\beta\|_2\le\|\beta-\bar\beta\|_2.
	\]
	Set $z^\pi:=X\beta^\pi-y$. Then
	\[
	\|z^\pi-\bar z\|_2\le\|X\|_2\|\beta^\pi-\bar\beta\|_2\le\|X\|_2\varepsilon_1\le\frac{\|\bar z\|_2}{2},
	\]
	so $z^\pi\in V$. Set $w^\pi:=g(z^\pi)$. The optimality of $\beta^\pi$ gives
	\[
	0\in\partial p(\beta^\pi)+X^\top w^\pi,
	\qquad w^\pi\in\tfrac1{\sqrt n}\partial\|z^\pi\|_2,
	\]
	so $u^\pi:=(w^\pi,\beta^\pi,z^\pi)\in S_{\mathrm{KKT}}$. Moreover,
	\[
	\begin{aligned}
		\|z-z^\pi\|_2 &\le r_X(u)+\|X\|_2\|\beta-\beta^\pi\|_2  \le C_z\rho(u),\qquad C_z:=1+\|X\|_2\kappa_{\mathrm P}C,
		\\
		\|w-w^\pi\|_2 &\le r_2(u)+L_g\|z-z^\pi\|_2 \le C_w\rho(u),\qquad C_w:=1+L_g C_z.
	\end{aligned}
	\]
	Consequently,
	\[
	\operatorname{dist}(u,S_{\mathrm{KKT}})
	\le\|u-u^\pi\|_2
	\le\sqrt{(\kappa_{\mathrm P}C)^2+C_z^2+C_w^2}\,\rho(u).
	\]
	
	\medskip\noindent\textbf{\textit{Case 2}: $\bar z=0$ and $\|\bar w\|_2<1/\sqrt n$.} 
	Let $d:=1/\sqrt n-\|\bar w\|_2>0$. For $\|w-\bar w\|_2<d/3$ and $\|v_2\|_2<d/3$, we have
	\[
	\|w+v_2\|_2
	\le \|\bar w\|_2+\|w-\bar w\|_2+\|v_2\|_2
	<\frac1{\sqrt n}.
	\]
	Since the subdifferential of $\|\cdot\|_2/\sqrt n$ is the closed ball of radius $1/\sqrt n$ at zero and consists of a vector of norm $1/\sqrt n$ at each nonzero point, the inclusion
	\[
	v_2\in\tfrac1{\sqrt n}\partial\|z\|_2-w
	\]
	holds if and only if $z=0$.
	Define
	\[
	\widehat{\mathcal T}(u):=
	\{z-X\beta+y\}\times\bigl(X^\top w+\partial p(\beta)\bigr)\times N_{\{0\}}(z),
	\qquad
	N_{\{0\}}(z)=
	\begin{cases}
		\mathbb R^n,&z=0,\\
		\varnothing,&z\ne0,
	\end{cases}
	\]
	where $N_{\{0\}}$ is the normal cone mapping to $\{0\}$. For $v=(v_X,v_p,v_2)$, the preceding equivalence yields neighborhoods $\mathcal U$ of $\bar u$ and $\mathcal V$ of $0$ such that
	\begin{equation}\label{eq:kkt-local-polyhedral-graph}
		\operatorname{gph}\mathcal T\cap(\mathcal U\times\mathcal V)
		=\operatorname{gph}\widehat{\mathcal T}\cap(\mathcal U\times\mathcal V).
	\end{equation}
	Since $p$ is polyhedral convex, $\partial p$ is a polyhedral multifunction; see \cite[Section~4.2]{zhou2017unified}. The remaining relations defining $\operatorname{gph}\widehat{\mathcal T}$ are affine, so this graph, and hence the graph of $\widehat{\mathcal T}^{-1}$, is a finite union of polyhedral convex sets. 
	By \cite[Proposition~3H.1]{dontchev2009implicit}, $\widehat{\mathcal T}^{-1}$ is calm at $0$ for $\bar u$. Thus, there exist constants $H>0$, $\delta_u>0$, and $\delta_v>0$ such that, whenever
	\[
	\|u-\bar u\|_2\le\delta_u,\qquad
	\|v\|_2\le\delta_v,\qquad
	v\in\widehat{\mathcal T}(u),
	\]
	there exists $u^0\in\widehat{\mathcal T}^{-1}(0)$ satisfying
	\[
	\|u-u^0\|_2\le H\|v\|_2.
	\]
	Choose $\delta_u$ and $\delta_v$ small enough that
	\[
	\{u:\|u-\bar u\|_2\le\delta_u\}\subseteq\mathcal U,
	\qquad
	\{v:\|v\|_2\le\delta_v\}\subseteq\mathcal V.
	\]
	For such $u$ and $v$ with $v\in\mathcal T(u)$, \eqref{eq:kkt-local-polyhedral-graph} gives $v\in\widehat{\mathcal T}(u)$. The corresponding point $u^0$ satisfies
	\[
	\|u^0-\bar u\|_2
	\le \|u^0-u\|_2+\|u-\bar u\|_2
	\le H\|v\|_2+\|u-\bar u\|_2\le H\delta_v+\delta_u.
	\]
	Take $\delta_u>0$ and $\delta_v>0$ sufficiently small so that $u^0\in\mathcal U$. Applying
	\eqref{eq:kkt-local-polyhedral-graph} at $(u^0,0)$ gives
	$u^0\in S_{\mathrm{KKT}}$. Thus,
	\[
	\operatorname{dist}(u,S_{\mathrm{KKT}})
	\le H\|v\|_2.
	\]
	Therefore $\mathcal T^{-1}$ is calm at $0$ for $\bar u$.
	By \cite[Theorem~3H.3]{dontchev2009implicit},
	$\mathcal T$ is metrically subregular at $\bar u$ for $0$,
	which gives \eqref{eq:kkt-local-eb}.
\end{proof}

Theorem~\ref{thm:kkt-eb} provides the local KKT error bound needed for the rate analysis. Combining this bound with the global convergence established above, we obtain the following local convergence rate of N-ALM.
\begin{theorem}[Local convergence rate of N-ALM]
	\label{thm:palm}
	Let $\{(w^k,\beta^k,z^k)\}$ be the sequence generated by
	Algorithm~\ref{Algorithm 4.1}, with
	\[
	0<\sigma_k\uparrow\sigma_\infty\le+\infty,
	\qquad
	\tau_k\downarrow\tau_\infty>0.
	\]
	Write $u^k:=(w^k,\beta^k,z^k)$, with $\mathcal M_k$ given by \eqref{block diagonal matrix}. Suppose that both stopping criteria \eqref{eq:nalm-stop-a} and \eqref{eq:nalm-stop-b} hold. Let $\bar u=(\bar w,\bar\beta,\bar z)\in S_{\mathrm{KKT}}$ be the limit established above, and suppose that either
	\[
	\bar z\ne0,
	\qquad\text{or}\qquad
	\bar z=0
	\quad\text{and}\quad
	\|\bar w\|_2<\frac{1}{\sqrt n}.
	\]
	Let $\kappa>0$ be a local error-bound constant supplied by
	Theorem~\ref{thm:kkt-eb} at $\bar u$.
	Then there exists $k_0$ such that, for all $k\ge k_0$,
	\begin{align}
		\operatorname{dist}_{\mathcal M_k}
		(u^{k+1},S_{\mathrm{KKT}})
		&\le
		\mu_k\,
		\operatorname{dist}_{\mathcal M_k}
		(u^k,S_{\mathrm{KKT}}),
		\label{eq:palm-local-contraction}
	\end{align}
	where
	\[
	\gamma_k:=\max\{1,\tau_k\},
	\qquad
	\mu_k:=
	\frac{1}{1-\delta_k}
	\left(
	\delta_k+
	\frac{(1+\delta_k)\kappa\gamma_k}
	{\sqrt{\sigma_k^2+\kappa^2\gamma_k^2}}
	\right).
	\]
	With $\gamma_\infty:=\max\{1,\tau_\infty\}$,
	\begin{equation}\label{eq:palm-contraction-limit}
		\lim_{k\to\infty}\mu_k
		=
		\frac{\kappa\gamma_\infty}
		{\sqrt{\sigma_\infty^2+\kappa^2\gamma_\infty^2}}.
	\end{equation}
	
\end{theorem}

\begin{proof}
	By \cite[Lemma~2.1 and Proposition~2.2(c)]{li2020asymptotically}, $\mathcal P_k$ is defined everywhere and is nonexpansive in the $\mathcal M_k$-norm. Also, $\mathcal P_k(v)=v$ for every $v\in S_{\mathrm{KKT}}$.
	
	Following the argument of \cite[Proposition~3.2]{li2020asymptotically}, we relate the stopping criteria to the proximal error. By \eqref{eq:nalm-multiplier-updates} and \eqref{sub_equ_problem},
	\[
	e^{k+1}:=(\nabla\psi_k(w^{k+1}),0,0)
	\in\mathcal T(u^{k+1})+\sigma_k^{-1}\mathcal M_k(u^{k+1}-u^k),
	\]
	and hence $u^{k+1}=\mathcal P_k(u^k+\sigma_k\mathcal M_k^{-1}e^{k+1})$. Nonexpansiveness gives
	\begin{equation}\label{eq:palm-proximal-error}
		E_k=\|u^{k+1}-\widetilde u^{k+1}\|_{\mathcal M_k}
		\le\sigma_k\|e^{k+1}\|_{\mathcal M_k^{-1}}
		=\frac{\sigma_k}{\sqrt{\tau_k}}\|\nabla\psi_k(w^{k+1})\|_2.
	\end{equation}
	Since $\min\{1,\sqrt{\tau_k}\}/\sqrt{\tau_k}\le1$, criterion \eqref{eq:nalm-stop-b} implies
	\begin{equation}\label{eq:palm-relative-proximal-error}
		E_k\le\delta_k\|u^{k+1}-u^k\|_{\mathcal M_k}.
	\end{equation}
	
	We follow the argument of \cite[Theorem~2.5]{li2020asymptotically} using the local error bound from Theorem~\ref{thm:kkt-eb}. Specifically, for some $r>0$,
	\begin{equation}\label{eq:palm-local-kkt-eb}
		\operatorname{dist}(u,S_{\mathrm{KKT}})\le\kappa\operatorname{dist}(0,\mathcal T(u))
		\qquad\text{whenever }\|u-\bar u\|_2\le r.
	\end{equation}
	Since $\mathcal P_k(\bar u)=\bar u$, nonexpansiveness and \eqref{eq:palm-metric-bounds} give
	\[
	\|\widetilde u^{k+1}-\bar u\|_2
	\le\underline m^{-1/2}\|\mathcal P_k(u^k)-\mathcal P_k(\bar u)\|_{\mathcal M_k}
	\le\sqrt{\frac{\gamma_0}{\underline m}}\|u^k-\bar u\|_2.
	\]
	Since $u^k$ converges to $\bar u$, there exists $k_0$ such that $\|u^k-\bar u\|_2\le r\sqrt{\underline m/\gamma_0}$ for every $k\ge k_0$. Hence \eqref{eq:palm-local-kkt-eb} applies to $\widetilde u^{k+1}$ for all such $k$.
	
	By the definition of $\mathcal P_k$,
	\begin{equation}\label{eq:palm-exact-proximal-inclusion}
		\sigma_k^{-1}\mathcal M_k(u^k-\widetilde u^{k+1})\in\mathcal T(\widetilde u^{k+1}).
	\end{equation}
	Combining this inclusion with \eqref{eq:palm-local-kkt-eb} and $\lambda_{\max}(\mathcal M_k)=\gamma_k$ gives, for $k\ge k_0$,
	\begin{equation}\label{eq:palm-exact-distance-bound}
		\operatorname{dist}_{\mathcal M_k}(\widetilde u^{k+1},S_{\mathrm{KKT}})
		\le\frac{\kappa\sqrt{\gamma_k}}{\sigma_k}\|\mathcal M_k(u^k-\widetilde u^{k+1})\|_2
		\le\frac{\kappa\gamma_k}{\sigma_k}\|u^k-\widetilde u^{k+1}\|_{\mathcal M_k}.
	\end{equation}
	Estimate \eqref{eq:palm-exact-distance-bound} is the error-bound estimate at the exact proximal points used in the proof of \cite[Theorem~2.5]{li2020asymptotically}. Combining it with \eqref{eq:palm-relative-proximal-error} and the firm nonexpansiveness inequality in \cite[Proposition~2.2(c)]{li2020asymptotically}, the same argument gives \eqref{eq:palm-local-contraction} with the stated $\mu_k$, taking the metric-change sequence $\nu_k=0$. This argument only uses the error bound at $\widetilde u^{k+1}$, which is ensured above for $k\ge k_0$. Finally, $\delta_k\to0$, $\gamma_k\to\gamma_\infty$, and $\sigma_k\uparrow\sigma_\infty$ give \eqref{eq:palm-contraction-limit}.
\end{proof}

\noindent\textit{Remark.}
By \eqref{eq:palm-contraction-limit}, $\lim_{k\to\infty}\mu_k<1$, which gives Q-linear convergence of the distance $\operatorname{dist}_{\mathcal M_k}(u^k,S_{\mathrm{KKT}})$. If $\sigma_\infty=+\infty$, then $\mu_k\to0$, giving Q-superlinear convergence, with finite termination allowed. In practice, however, the penalty parameter is often bounded above to avoid numerical difficulties caused by excessively large values.
\section{Numerical Experiments}\label{numersec}
In this section, we conduct numerical experiments to evaluate the computational performance of N-ALM for the SRCL problem \eqref{square_root_clustered_lasso_model}. We first report its accuracy and running time on high-dimensional real-world datasets and then compare its computational performance with d-ADMM for the same SRCL problem.

It is worth noting that the SRCL problem \eqref{square_root_clustered_lasso_model} is considered for the first time in this paper, and no publicly available algorithms are currently available for performance comparison. To address this dilemma, we further develop two alternating direction method of multipliers (ADMM) algorithms for directly solving the SRCL problem \eqref{square_root_clustered_lasso_model} and its dual problem, respectively. We use the prefixes ``p-'' and ``d-'' to distinguish the primal and dual implementations, respectively.
In the following numerical experiments, the tuning parameters $\lambda_1$ and $\lambda_2$ in the SRCL model \eqref{square_root_clustered_lasso_model} are chosen as
\begin{equation}\label{alpha12}
	\lambda_1 = \alpha_1 \frac{\|X^\top y\|_\infty}{\|y\|_2},
	\quad\mathrm{and}\quad
	\lambda_2 = \alpha_2\lambda_1,
\end{equation}
where $0<\alpha_1<1$ and $\alpha_2>0$. This is an empirical parameter rule for computational benchmarking; it is not the conditional-quantile tuning rule in Theorem~\ref{thm:pivotal_tuning}, and the finite-sample guarantees do not automatically apply to these parameter choices. The factor $1/\sqrt n$ in the normalized score is absorbed into the empirical choice of $\alpha_1$. In particular, to obtain reasonable clustering results, we set $\alpha_2=O(1/p)$ so that the two penalty terms have comparable magnitudes.
For the outer proximal augmented Lagrangian iteration (Algorithm \ref{Algorithm 4.1}) and the inner semismooth Newton solver (Sub-Algorithm \ref{Algorithm 4.2}), the maximum number of iterations is set to 200. For the ADMM-type methods, the maximum number of iterations is set to 20,000. All algorithms terminate when the prescribed stopping tolerance is reached or when the total CPU time exceeds 3 hours. All numerical experiments are conducted on a personal computer equipped with an AMD Ryzen 5 9600X CPU with a base frequency of 3.90 GHz and 32 GB of RAM, running Microsoft Windows 10 with MATLAB R2024a.

\subsection{Formulations of p-ADMM and d-ADMM}\label{pdadmm}
In this subsection, we briefly introduce two variants of ADMM for solving the SRCL model \eqref{square_root_clustered_lasso_model}, which are derived from the primal formulation \eqref{p_problem} and dual formulation \eqref{d_problem}, respectively.
The d-ADMM serves as a baseline to verify the computational efficiency of the proposed N-ALM algorithm. 
For completeness, we also present p-ADMM for the primal formulation \eqref{p_problem}; the numerical comparison below concerns N-ALM and d-ADMM.

By introducing slack variables $\alpha \in \mathbb{R}^p$, the primal form of problem  \eqref{p_problem} can be rewritten equivalently as
$$
\text{(P)} \quad\min_{\beta \in \mathbb{R}^p, z \in \mathbb{R}^n, \alpha \in \mathbb{R}^p} \Big\{\frac{1}{\sqrt{n}}\|z\|_2 + p(\alpha) \ \mid \ \beta - \alpha = 0,\ X\beta - y = z \Big\},
$$
and its dual formulation \eqref{d_problem} takes the following equivalent  form:
$$
\text{(D)} \quad \min_{w \in \mathbb{R}^n, s \in \mathbb{R}^p, q \in \mathbb{R}^n} \Big\{ p^*(s)+\delta_{B_2^{(1/\sqrt{n})}}(q) +\langle w, y \rangle \ \mid \ X^{\top} w + s = 0, \ w - q = 0 \Big\}.
$$
Let $\sigma > 0$ be given. The corresponding augmented Lagrangian function associated with (P) is given by
$$
{L}_{\sigma}(\beta, z, \alpha; s, w) = \frac{1}{\sqrt{n}}\|z\|_2 + p(\alpha) + \langle s, \beta - \alpha \rangle + \frac{\sigma}{2} \|\beta - \alpha\|_2^2 + \langle w, X\beta - y - z \rangle + \frac{\sigma}{2} \|X\beta - y - z\|_2^2,
$$
where $w \in \mathbb{R}^n$ and $s \in \mathbb{R}^p$ are multipliers. By direct calculation, the iterative steps of the algorithm for Problem (P) are obtained as follows.

\begin{framed} 
	\noindent  
	\textbf{Algorithm 6.1: p-ADMM: An alternating direction method of multipliers for (P)}
	\label{Algorithm 6.1} 
	\vskip 3.0mm \hrule \vskip 3mm 
	
	\noindent{\textbf{Initialization.}} $\kappa\in(0,(1+\sqrt{5})/2)$, $\sigma>0$; initialize $z^0\in\mathbb{R}^n$, $\alpha^0\in\mathbb{R}^p$, 
	$s^0\in\mathbb{R}^p$, $w^0\in\mathbb{R}^n$, and let $k:=0$. Do:
	\vskip 2mm 
	
	\noindent{\textbf{Step 1.}} Compute	$\beta^{k+1}:=\underset{\beta\in\mathbb{R}^p}{\operatorname{argmin}} \Big\{ {L}_{\sigma}(\beta,z^k,\alpha^k;s^k,w^k)$\Big\};
	\vskip 2mm 
	
	\noindent{\textbf{Step 2.}} Compute	$z^{k+1}=\underset{z\in\mathbb{R}^n}{\operatorname{argmin}}\Big\{{L}_{\sigma}(\beta^{k+1},z,\alpha^k;s^k,w^k)\Big\}
	=\operatorname{Prox}_{\frac{1}{\sqrt{n}\sigma}\|\cdot\|_2}\left(X\beta^{k+1}-y+\frac{w^k}{\sigma}\right)$;
	\vskip 2mm 
	
	\noindent{\textbf{Step 3.}} Compute	$\alpha^{k+1}=\underset{\alpha\in\mathbb{R}^p}{\operatorname{argmin}} \Big\{ {L}_{\sigma}(\beta^{k+1},z^{k+1},\alpha;s^k,w^k)\Big\}
	=\operatorname{Prox}_{p(\cdot)/\sigma}\left(\beta^{k+1}+\frac{s^k}{\sigma}\right)$;
	\vskip 2mm 
	
	\noindent{\textbf{Step 4.}} Update	$s^{k+1}=s^k+\kappa\sigma(\beta^{k+1}-\alpha^{k+1})$ and $w^{k+1}=w^k+\kappa\sigma(X\beta^{k+1}-y-z^{k+1})$;
	\vskip 2mm 
	
	\noindent{\textbf{Step 5.}} 
	Let $k:=k+1$ and go to Step 1.
\end{framed}

It is straightforward to verify that solving the $\beta$-subproblem in
\textbf{Step 1} is equivalent to solving the following $p$-dimensional
linear system:
$$
\Big(\sigma I_p + \sigma X^{\top} X\Big) \beta=X^{\top} \Big(\sigma (y + z^k)-w^k\Big)+\sigma \alpha^k-s^k,
$$
where $I_p$ denotes the identity matrix in $\mathbb{R}^p$. This linear system can be efficiently solved using standard linear system solvers,
especially in the low-dimensional setting with $p\ll n$, and we omit the implementation details for brevity.

Let $\sigma>0$ be given. The corresponding augmented Lagrangian function
associated with the dual problem (D) is given by
$$
	\widetilde{L}_{\sigma}(w,s,q;\beta,z)={}p^*(s)+ \delta_{B_2^{(1/\sqrt{n})}}(q)+ \langle w,y\rangle- \langle \beta,X^{\top}w+s\rangle 
	+ \frac{\sigma}{2}\left\|X^{\top}w+s\right\|_2^2+ \langle z,w-q\rangle+ \frac{\sigma}{2}\left\|w-q\right\|_2^2,
$$
where $\beta\in\mathbb{R}^p$ and $z\in\mathbb{R}^n$ are the Lagrange
multipliers, with $z$ following the same sign convention $z=X\beta-y$ as in \eqref{p_problem}. By direct calculation, the iterative scheme for solving
Problem (D) is given as follows.

\begin{framed} 
	\noindent  
	\textbf{Algorithm 6.2: d-ADMM: An alternating direction method of multipliers for (D)}
	\label{Algorithm 6.2} 
	\vskip 3.0mm \hrule \vskip 3mm 
	
	\noindent{\textbf{Initialization.}} $\kappa\in(0,(1+\sqrt{5})/2)$, $\sigma>0$; initialize 
	$w^0\in\mathbb{R}^n$, $s^0\in\mathbb{R}^p$, $q^0\in\mathbb{R}^n$, 
	$\beta^0\in\mathbb{R}^p$, $z^0\in\mathbb{R}^n$. Let $k:=0$. Do:
	\vskip 2mm 
	
	\noindent{\textbf{Step 1.}} Compute	$w^{k+1}:=\underset{w\in\mathbb{R}^n}{\operatorname{argmin}} \Big\{\widetilde{L}_{\sigma}(w,s^k,q^k;\beta^k,z^k)\Big\}$;
	\vskip 2mm 
	
	\noindent{\textbf{Step 2.}} Compute	$s^{k+1}:=\underset{s\in\mathbb{R}^p}{\operatorname{argmin}}	\Big\{\widetilde{L}_{\sigma}(w^{k+1},s,q^k;\beta^k,z^k)\Big\}
	=\operatorname{Prox}_{p^*(\cdot)/\sigma}\left(\frac{\beta^k}{\sigma}-X^{\top}w^{k+1}\right)$;
	\vskip 2mm 
	
	\noindent{\textbf{Step 3.}} Compute	$q^{k+1}:=\underset{q\in\mathbb{R}^n}{\operatorname{argmin}}	\Big\{\widetilde{L}_{\sigma}(w^{k+1},s^{k+1},q;\beta^k,z^k)\Big\}
	=\operatorname{Proj}_{B_2^{(1/\sqrt{n})}}\left(	w^{k+1}+\frac{z^k}{\sigma}\right)$;
	\vskip 2mm 
	
	\noindent{\textbf{Step 4.}} Update	$\beta^{k+1}=\beta^k-\kappa\sigma(X^{\top}w^{k+1}+s^{k+1})$ and $z^{k+1}=z^k+\kappa\sigma(w^{k+1}-q^{k+1})$;
	\vskip 2mm 
	
	\noindent{\textbf{Step 5.}} 
	Let $k:=k+1$ and go to Step 1. 
\end{framed}

In a similar manner, the $w$-subproblem can be reduced to solving the following $n$-dimensional linear system:
$$
\big(\sigma I_n+\sigma XX^{\top}\big)w=X\beta^k-y-z^k+\sigma q^k-\sigma Xs^k,
$$
which can be efficiently solved using standard linear system solvers, especially when $n\ll p$. The convergence properties of p-ADMM and d-ADMM follow from the convergence analysis established in Theorem B1 of \cite{fazel2013hankel}. Therefore, we omit the convergence analysis for these two algorithms.

\subsection{Stopping criteria}
In this part, we focus on constructing the stopping criteria implemented in Algorithm \ref{Algorithm 4.1}.
We can deduce that the primal-dual Karush--Kuhn--Tucker (KKT) conditions associated with problems \eqref{p_problem} and \eqref{d_problem} can be expressed as
\begin{equation}\label{kkt_p}
	y+z-X\beta=0,\quad 
	-X^\top w\in\partial p(\beta),\quad
	w\in\partial\frac{1}{\sqrt{n}}\|z\|_2.
\end{equation}
Consequently, we can define the relative KKT residual as follows:
\begin{equation} \label{eq:eta_kkt}
	\eta_{\text{kkt}} = \max\Big\{\eta_\beta, \eta_z, \eta_w\Big\}, 
\end{equation}
where
\begin{align*}
	\eta_\beta &= \frac{\| -X^{\top} w+ \operatorname{Prox}_{p^{*}}(\beta -X^{\top} w)\|_2}{1 + \|\beta\|_2 + \|X^{\top} w\|_2}, &
	\eta_z &= \frac{\|w - \Pi_{B_2^{(1/\sqrt{n})}}(z+w)\|_2}{1 + \|z\|_2 + \|w\|_2}, &
	\eta_w &= \frac{\|X\beta - z - y\|_2}{1 + \|y\|_2} .
\end{align*}

On the one hand, for the equivalent formulation (P) of primal problem \eqref{p_problem}, we see that the relative KKT residual \eqref{eq:eta_kkt} takes the following equivalent form
\begin{equation}\label{PADMM stop}
	\eta_{p} = \max\Big\{\eta_1,\eta_2,\eta_3,\eta_4,\eta_5\Big\},
\end{equation}
where
\begin{align*}
	\eta_1 &= \frac{\|X\beta - z - y\|_2}{1 + \|y\|_2}, &
	\eta_2 &= \frac{\|\beta - \alpha\|_2}{1 + \|\beta\|_2 + \|\alpha\|_2}, &
	\eta_3 &= \frac{\|X^{\top}w + s\|_2}{1 + \|X^{\top}w\|_2 + \|s\|_2} , &
\end{align*}
\begin{align*}
	\eta_4 &= \frac{\|z - \operatorname{Prox}_{\frac{1}{\sqrt{n}}\|\cdot\|_{2}}(z+w)\|_2}{1 + \|z\|_2 + \|w\|_2} , &
	\eta_5 &= \frac{\|\alpha - \operatorname{Prox}_{p(\cdot)}(\alpha + s)\|_2}{1 + \|\alpha\|_2 + \|s\|_2} . &
\end{align*}
On the other hand, for the equivalent formulation (D) of dual problem \eqref{d_problem}, its
relative KKT residual can be reorganized as follows:
\begin{equation}\label{DADMM stop} 
	\eta_{d} = \max\big\{\eta_1,\eta_2,\eta_3,\eta_4,\eta_5\big\},
\end{equation}
where
\begin{align*}
	\eta_1 &= \frac{\|X^{\top}w + s\|_2}{1 + \|X^{\top}w\|_2 + \|s\|_2}, &
	\eta_2 &= \frac{\|w - q\|_2}{1 + \|w\|_2 + \|q\|_2}, &
	\eta_3 &= \frac{\|X\beta - z - y\|_2}{1 + \|y\|_2} , &
\end{align*}
\begin{align*}
	\eta_4 &= \frac{\|s - \operatorname{Prox}_{p^*(\cdot)}(s+\beta)\|_2}{1 + \|s\|_2 + \|\beta\|_2} , &
	\eta_5 &= \frac{\|q - \Pi_{B_2^{(1/\sqrt{n})}}(q + z)\|_2}{1 + \|q\|_2 + \|z\|_2} . &
\end{align*}
Hence, in the following tests, it is reasonable to terminate the iterative process of Algorithms \ref{Algorithm 4.1}, \ref{Algorithm 6.1}, and \ref{Algorithm 6.2} such that the measurements $\eta_{\text{KKT}}$, $\eta_p$, and $\eta_d$ are sufficiently small. 

In addition to the above relative KKT residuals $\eta_{\mathrm{KKT}}$, we further evaluate the accuracy of the solutions by monitoring the relative duality gap between the primal problem (P) and dual problem (D), which is defined as
\begin{equation}\label{eq:gap}
	\eta_{\text{gap}}:=\frac{p(\beta)+\frac{1}{\sqrt{n}}\|X\beta-y\|_2+p^*(s)+\langle w,y\rangle}{1+p(\beta)+\frac{1}{\sqrt{n}}\|X\beta-y\|_2+|p^*(s)|+|\langle w,y\rangle|}.
\end{equation}
The penalty parameters $\sigma_k$ and $\tau_k$ in Algorithm \ref{Algorithm 4.1} are updated according to
$$
	\sigma_{k+1} = 
	\begin{cases} 
		\sigma_k, & \text{if } \text{mod}(k, 3) \neq 0 \\ 
		\min\{3\sigma_k, 10^6\}, & \text{otherwise} 
	\end{cases}
	\quad\text{and}\quad \tau_{k+1} = 
	\begin{cases} 
		\tau_k & \text{if } \text{mod}(k, 3) \neq 0 \\ 
		\max\{2\tau_k/3, 10^{-6}\}, & \text{otherwise}. 
	\end{cases}
$$
Such adaptive updating strategies for penalty parameters are widely adopted in augmented Lagrangian methods. We therefore omit further discussion for brevity.

\subsection{Numerical Results on UCI Data Sets}
We adopt several datasets from the UCI (University of California, Irvine) Machine Learning Repository, which are downloaded in LIBSVM format \citep{chang2011libsvm}  from the \href{https://www.csie.ntu.edu.tw/~cjlin/libsvmtools/datasets/}{LIBSVM dataset collection}.
The selected datasets include E2006, pyrim5, triazines4, abalone, bodyfat, and housing, all of which correspond to high-dimensional low-sample-size settings.
For a given tolerance $\epsilon$, we terminate the iterative process of N-ALM when
\begin{equation}\label{ssnal stop} 
	\max\big\{\eta_{\text{gap}}, \eta_{\text{kkt}}\big\} \le \epsilon.
\end{equation}
The detailed numerical results under different settings are reported in Tables \ref{tab:uci_results} and \ref{tab:performance_updated_v4}.

In Table \ref{tab:uci_results}, the first column contains the problem names, where $n$ is relatively small and $p$ is relatively large. The notation $\lambda_{\max}(XX^\top)$ represents the largest eigenvalue of $XX^\top$. The second column reports the values of $\alpha_1$ and $\alpha_2$ used to determine the tuning parameters defined in \eqref{alpha12}. Specifically, to obtain a reasonable number of nonzero elements in the optimal solution $\beta$, we choose $\alpha_1 \in \{10^{-5}, 10^{-6}\}$ for the E2006.train and E2006.test instances, $\alpha_1 \in \{10^{-4}, 10^{-5}\}$ for the bodyfat instance, and $\alpha_1 \in \{10^{-2}, 10^{-3}\}$ for the remaining instances. In total, there are $42$ tested cases in the first test, as reported in Table \ref{tab:uci_results}.
The third column reports several metrics associated with the optimal solution $\widehat{\beta}$. We define $\text{nnz}(\widehat{\beta})$ as
\begin{equation} 
	\text{nnz}(\widehat{\beta}) := \min \Big\{ k \,\Big|\, \sum_{i=1}^k |\widetilde{\beta}_i| \ge 0.99999 \|\widehat{\beta}\|_1 \Big\},
\end{equation}
where $\widetilde{\beta}$ denotes the vector obtained by sorting the entries of $\widehat{\beta}$ in descending order of magnitude, i.e., $|\widetilde{\beta}_1| \ge |\widetilde{\beta}_2| \ge \dots \ge |\widetilde{\beta}_p|$. The symbol $\text{gnnz}(\widehat{\beta})$ denotes the number of groups identified from the solution. Within each group, the ratio of the absolute values of any two entries is required to lie in the interval $[5/6,6/5]$. In addition, all entries whose absolute values are below $10^{-4}$ are assigned to a single common group.
Moreover, the fourth column reports the objective function values of \eqref{square_root_clustered_lasso_model}, the fifth column reports $\eta_{\text{kkt}}$ defined in \eqref{eq:eta_kkt}, the sixth column reports $\eta_{\text{gap}}$ defined in \eqref{eq:gap}, and the last column reports the running time in seconds, with the time presented in hours:minutes:seconds format.

  	\setlength{\tabcolsep}{6pt}
  	\renewcommand{\arraystretch}{0.9}
  	 
  	\begin{longtable}{|>{\raggedright\arraybackslash}p{4.2cm}|c|c|r|r|r|r|}
  		\caption{Numerical results of N-ALM on UCI data sets with different $\alpha_1$ and $\alpha_2$ under the termination criterion $\max\{ \eta_{\text{gap}},  \eta_{\text{kkt}}\} \leq 10^{-6}$. } \label{tab:uci_results} \\
  		\hline
  		Pro-name ($n; p$) & $\alpha_1; \alpha_2$ & nnz($\beta$); gnnz($\beta$) & pobj & $\eta_{\text{kkt}}$ & $ \eta_{\text{gap}}$ & time \\ \hline
  		\endfirsthead
  		\multicolumn{7}{c}{\bfseries \tablename\ \thetable{} -- continued} \\
  		\hline
  		Pro-name ($n; p$) & $\alpha_1; \alpha_2$ & nnz($\beta$); gnnz($\beta$) & pobj & $\eta_{\text{kkt}}$ & $ \eta_{\text{gap}}$ & time \\ \hline
  		\endhead
  		\hline
  		\multicolumn{7}{|r|}{{Continued on the next page}} \\
  		\hline
  		\endfoot
  		\hline
  		\endlastfoot
  		
  		\multirow{6}{=}{E2006.train (16087; 150360)\\ $\lambda_{\max}(XX^\top) =1.91\text{e}+05$}
  		& 1e$-5$; 1e$-5$ & 4; 4 & 4.88110e+1 & 6.9e-7 & 6.9e-7 & 11 \\*
  		& 1e$-5$; 1e$-6$ & 15; 5 & 4.86127e+1 & 7.4e-8 & 7.8e-7 & 19 \\*
  		& 1e$-5$; 1e$-7$ & 24; 5 & 4.85595e+1 & 2.0e-7 & 8.0e-7 & 19 \\* \cline{2-7}
  		& 1e$-6$; 1e$-4$ & 7; 7 & 4.87243e+1 & 2.2e-7 & 8.8e-7 & 16 \\*
  		& 1e$-6$; 5e$-5$ & 31; 16 & 4.84618e+1 & 1.1e-7 & 9.9e-7 & 23 \\*
  		& 1e$-6$; 1e$-5$ & 356; 8 & 4.70894e+1 & 7.8e-9 & 8.4e-7 & 57 \\ \hline
  		
  		\multirow{6}{=}{E2006.test (3308; 150358)\\ $\lambda_{\max}(XX^\top) =4.79\text{e}+04$}
  		& 1e$-5$; 1e$-5$ & 10; 6 & 2.18509e+1 & 2.1e-8 & 8.0e-7 & 10 \\*
  		& 1e$-5$; 1e$-6$ & 42; 9 & 2.13824e+1 & 7.0e-9 & 8.6e-7 & 13 \\*
  		& 1e$-5$; 1e$-7$ & 61; 5 & 2.12543e+1 & 6.5e-9 & 8.8e-7 & 14 \\* \cline{2-7}
  		& 1e$-6$; 1e$-4$ & 21; 5 & 2.16349e+1 & 4.8e-9 & 9.6e-7 & 12 \\*
  		& 1e$-6$; 5e$-5$ & 86; 6 & 2.10405e+1 & 8.1e-9 & 8.7e-7 & 16 \\*
  		& 1e$-6$; 1e$-5$ & 980; 6 & 1.74579e+1 & 1.4e-9 & 8.5e-7 & 1:23 \\ \hline
  		
  		\multirow{6}{=}{pyrim5 (74; 201376)\\ $\lambda_{\max}(XX^\top) =1.22\text{e}+06$}
  		& 1e$-2$; 5e$-5$ & 69; 5 & 1.29260e+0 & 2.8e-7 & 9.7e-7 & 22 \\*
  		& 1e$-2$; 1e$-5$ & 88; 9 & 4.78402e-1 & 2.5e-7 & 8.3e-7 & 27 \\*
  		& 1e$-2$; 1e$-6$ & 80; 8 & 2.42379e-1 & 8.2e-7 & 8.2e-7 & 11 \\* \cline{2-7}
  		& 1e$-3$; 5e$-5$ & 79; 6 & 2.29042e-1 & 7.5e-7 & 7.5e-7 & 22 \\*
  		& 1e$-3$; 1e$-5$ & 84; 5 & 1.11376e-1 & 9.9e-7 & 9.9e-7 & 21 \\*
  		& 1e$-3$; 1e$-6$ & 86; 6 & 8.37286e-2 & 9.9e-7 & 9.9e-7 & 16 \\ \hline
  		
  		\multirow{6}{=}{triazines4 (186; 635376)\\ $\lambda_{\max}(XX^\top) =2.07\text{e}+07$}
  		& 1e$-2$; 1e$-5$ & 360; 6  & 2.41031e+0 & 3.6e-7 & 5.3e-8 & 51 \\*
  		& 1e$-2$; 5e$-6$ & 589; 9  & 1.90339e+0 & 9.3e-7 & 2.7e-7 & 1:23 \\*
  		& 1e$-2$; 1e$-6$ & 1738; 6 & 1.25306e+0 & 6.3e-7 & 9.8e-7 & 3:52 \\* \cline{2-7}
  		& 1e$-3$; 1e$-5$ & 941; 6  & 8.15855e-1 & 9.9e-7 & 9.9e-7 & 3:15 \\*
  		& 1e$-3$; 5e$-6$ & 1088; 8 & 6.16552e-1 & 9.9e-7 & 9.9e-7 & 3:07 \\*
  		& 1e$-3$; 1e$-6$ & 2817; 9 & 4.18310e-1 & 5.6e-7 & 5.6e-7 & 4:21 \\ \hline
  		
  		\multirow{6}{=}{abalone (4177; 6435)\\ $\lambda_{\max}(XX^\top) =5.21\text{e}+05$}
  		& 1e$-2$; 1e$-4$ & 16; 6 & 1.74201e+2 & 2.9e-8 & 5.3e-7 & 7 \\*
  		& 1e$-2$; 5e$-5$ & 19; 7 & 1.68743e+2 & 4.2e-8 & 6.2e-7 & 7 \\*
  		& 1e$-2$; 1e$-5$ & 17; 7 & 1.64072e+2 & 1.9e-7 & 6.6e-7 & 8 \\* \cline{2-7}
  		& 1e$-3$; 1e$-4$ & 43; 5 & 1.41833e+2 & 6.3e-7 & 6.3e-7 & 24 \\*
  		& 1e$-3$; 5e$-5$ & 44; 8 & 1.40555e+2 & 9.2e-7 & 9.2e-7 & 22 \\*
  		& 1e$-3$; 1e$-5$ & 46; 7 & 1.39405e+2 & 8.5e-7 & 8.5e-7 & 23 \\ \hline
  		
  		\multirow{6}{=}{bodyfat (252; 116280)\\ $\lambda_{\max}(XX^\top) =5.29\text{e}+04$}
  		& 1e$-4$; 1e$-4$ & 84; 11  & 4.73015e-2 & 8.4e-7 & 8.4e-7 & 14 \\*
  		& 1e$-4$; 5e$-5$ & 159; 5  & 3.15971e-2 & 9.5e-7 & 9.5e-7 & 15 \\*
  		& 1e$-4$; 1e$-5$ & 237; 10 & 1.13149e-2 & 8.3e-7 & 9.0e-7 & 39 \\* \cline{2-7}
  		& 1e$-5$; 1e$-4$ & 239; 11 & 6.60526e-3 & 7.3e-7 & 7.3e-7 & 54 \\*
  		& 1e$-5$; 5e$-5$ & 240; 12 & 3.56418e-3 & 9.9e-7 & 9.9e-7 & 53 \\*
  		& 1e$-5$; 1e$-5$ & 244; 11 & 1.13143e-3 & 9.1e-7 & 9.1e-7 & 39 \\ \hline
  		
  		\multirow{6}{=}{housing (506; 77520)\\ $\lambda_{\max}(XX^\top) =3.28\text{e}+05$}
  		& 1e$-2$; 5e$-5$ & 89; 5   & 1.44559e+2 & 6.5e-7 & 7.1e-7 & 10 \\*
  		& 1e$-2$; 1e$-5$ & 142; 5  & 9.59228e+1 & 8.3e-7 & 8.3e-7 & 11 \\*
  		& 1e$-2$; 1e$-6$ & 156; 20 & 8.00068e+1 & 7.3e-7 & 7.3e-7 & 8 \\* \cline{2-7}
  		& 1e$-3$; 5e$-5$ & 213; 7  & 6.17610e+1 & 8.9e-7 & 8.9e-7 & 22 \\*
  		& 1e$-3$; 1e$-5$ & 286; 5  & 4.52568e+1 & 2.4e-7 & 7.4e-7 & 26 \\*
  		& 1e$-3$; 1e$-6$ & 339; 8  & 3.83327e+1 & 2.7e-7 & 9.3e-7 & 18 \\
  	\end{longtable}
  	  
    It can be seen from Table \ref{tab:uci_results} that N-ALM consistently achieves highly accurate solutions under different choices of $\alpha_1$ and $\alpha_2$, with both the KKT residual and the relative duality gap satisfying the termination criterion
\[
\max\{\eta_{\mathrm{gap}},\eta_{\mathrm{kkt}}\}\leq10^{-6}.
\]
Moreover, the computing time remains reasonable across datasets of different sizes and dimensions. We also see that most test instances can be solved within several tens of seconds, while the computing time remains below four minutes even for the more challenging cases. Overall, these results demonstrate the effectiveness and numerical robustness of N-ALM across different datasets and parameter settings.
     
    We have clarified in Subsection \ref{pdadmm} that, the ADMM method based on the dual formulation, referred to as d-ADMM and presented in Algorithm~\ref{Algorithm 6.2}, may be particularly effective for high-dimensional problems with $n\ll p$. 
    We therefore conduct additional experiments to compare the performance of d-ADMM and N-ALM.
    For each data set, we consider two representative parameter settings for $\alpha_1$ and $\alpha_2$ and report the detailed numerical results in Table~\ref{tab:performance_updated_v4}. The columns labeled ``Rel. Res.'', ``pobj'', ``iter'', and ``time'' report the relative residual, the primal objective value, the number of iterations, and the running time, respectively. Specifically, the ``Rel. Res.'' column corresponds to $\eta_{\mathrm{kkt}}$ in \eqref{eq:eta_kkt} for N-ALM and $\eta_d$ in \eqref{DADMM stop} for d-ADMM. For N-ALM, the total number of Newton iterations performed in the inner SSN procedure is additionally reported in parentheses.
    The algorithms are terminated when $\max\{\eta_{\mathrm{gap}},\eta_{\mathrm{kkt}}\}\leq 10^{-6}$ for N-ALM and $\eta_d\leq 10^{-6}$ for d-ADMM. In addition, both algorithms are terminated if the iteration count reaches $20{,}000$ before the corresponding stopping criterion is satisfied. In such cases, the corresponding entries are marked by ``---''.

     \setlength{\tabcolsep}{3.8pt}
     \renewcommand{\arraystretch}{1.2}
     
     \begin{longtable}{|>{\raggedright\arraybackslash}p{3.0cm}|c|r|r|r|r|r|r|r|r|}
     	\caption{Numerical results of N-ALM and d-ADMM on UCI data sets. In the table, ``a'' = N-ALM and ``b'' = d-ADMM. The termination criterion is $\max\{\eta_{\text{gap}},\eta_{\text{kkt}}\}\leq 10^{-6}$ for N-ALM and $\eta_d\leq 10^{-6}$ for d-ADMM.}
     	\label{tab:performance_updated_v4} \\
     	\hline
     	Pro-name & $\alpha_1;\alpha_2$ & \multicolumn{2}{r|}{Rel. Res.} & \multicolumn{2}{r|}{pobj} & \multicolumn{2}{r|}{iter} & \multicolumn{2}{r|}{time} \\
     	\cline{3-10}
     	&                    & a             & b             & a            & b            & a          & b          & a       & b       \\
     	\hline
     	\endfirsthead
     	
     	\multirow{2}{=}{E2006.train}
     	& $1\mathrm{e}{-5};\,1\mathrm{e}{-7}$ & 8.0e-7 & 9.6e-7 & 4.8560e+1 & 4.8560e+1 &  65(113) &  354 & 19    & 1:37 \\
     	& $1\mathrm{e}{-6};\,1\mathrm{e}{-5}$ & 8.4e-7 & 1.0e-6 & 4.7089e+1 & 4.7089e+1 & 124(266) &  875 & 57    & 2:59 \\
     	\hline
     	
     	\multirow{2}{=}{E2006.test}
     	& $1\mathrm{e}{-5};\,1\mathrm{e}{-7}$ & 8.8e-7 & 1.0e-6 & 2.1254e+1 & 2.1254e+1 & 102(190) &  771 & 14    & 33 \\
     	& $1\mathrm{e}{-6};\,1\mathrm{e}{-5}$ & 8.5e-7 & 1.0e-6 & 1.7458e+1 & 1.7458e+1 & 181(741) & 5174 & 1:23  & 3:27 \\
     	\hline
     	
     	\multirow{2}{=}{pyrim5}
     	& $1\mathrm{e}{-2};\,1\mathrm{e}{-6}$ & 8.2e-7 & 7.1e-6 & 2.4238e-1 & 2.4239e-1 & 38(132) & ---  & 11    & 17:20 \\
     	& $1\mathrm{e}{-3};\,1\mathrm{e}{-6}$ & 9.9e-7 & 3.6e-5 & 8.3729e-2 & 8.3732e-2 & 61(189) & ---  & 16    & 16:56 \\
     	\hline
     	
     	\multirow{2}{=}{triazines4}
     	& $1\mathrm{e}{-2};\,1\mathrm{e}{-6}$ & 9.8e-7 & 7.5e-6 & 1.2531e+0 & 1.2531e+0 & 69(435) & ---  & 3:52  & 1:14:21 \\
     	& $1\mathrm{e}{-3};\,1\mathrm{e}{-6}$ & 5.6e-7 & 9.3e-5 & 4.1831e-1 & 4.1837e-1 & 83(566) & ---  & 4:21  & 1:13:14 \\
     	\hline
     	
     	\multirow{2}{=}{abalone}
     	& $1\mathrm{e}{-2};\,1\mathrm{e}{-5}$ & 6.6e-7 & 9.9e-7 & 1.6407e+2 & 1.6407e+2 & 39(127) & 5195 & 8     & 2:20 \\
     	& $1\mathrm{e}{-3};\,1\mathrm{e}{-5}$ & 8.5e-7 & 1.0e-6 & 1.3941e+2 & 1.3941e+2 & 43(123) & 2240 & 23    & 59 \\
     	\hline
     	
     	\multirow{2}{=}{bodyfat}
     	& $1\mathrm{e}{-4};\,1\mathrm{e}{-5}$ & 9.0e-7 & 9.3e-5 & 1.1315e-2 & 1.1655e-2 & 54(537) & ---  & 39    & 11:40 \\
     	& $1\mathrm{e}{-5};\,1\mathrm{e}{-5}$ & 9.1e-7 & 1.1e-5 & 1.1314e-3 & 1.3031e-3 & 85(529) & ---  & 39    & 11:42 \\
     	\hline
     	
     	\multirow{2}{=}{housing}
     	& $1\mathrm{e}{-2};\,1\mathrm{e}{-6}$ & 7.3e-7 & 1.0e-6 & 8.0007e+1 & 8.0007e+1 & 44(148) & 3604 & 8     & 2:15 \\
     	& $1\mathrm{e}{-3};\,1\mathrm{e}{-6}$ & 9.3e-7 & 1.0e-6 & 3.8333e+1 & 3.8333e+1 & 77(249) & 9079 & 18    & 5:46 \\
     	\hline
     \end{longtable}
     
     It can be seen from Table \ref{tab:performance_updated_v4} that N-ALM exhibits a clear computational advantage over d-ADMM on the tested high-dimensional data sets. In particular, N-ALM obtains primal objective values comparable to those of d-ADMM, while generally achieving smaller relative residuals. The advantage of N-ALM is particularly evident in terms of computational efficiency. While d-ADMM may require a substantially larger number of iterations and, in several cases, reaches the prescribed iteration limit without satisfying its stopping criterion, N-ALM solves the corresponding problems within considerably shorter running times. This difference is especially evident for the larger-scale instances, where d-ADMM requires several minutes or even more than ten minutes, whereas N-ALM typically completes the computation within a much shorter time. Notably, for the pyrim5, triazines4, and bodyfat data sets, d-ADMM fails to satisfy its stopping criterion within the prescribed iteration limit, whereas N-ALM successfully solves all three instances in only a few tens of seconds. These results highlight the computational advantage of N-ALM for high-dimensional problems with $n\ll p$, particularly when d-ADMM struggles to reach the required solution accuracy within a reasonable computational budget.
     
	\section{ Conclusion}\label{con}
	In this paper, we proposed a square-root clustered Lasso penalized method for high-dimensional linear regression. By combining square-root loss and the clustered Lasso penalty, we obtained a flexible model with tuning parameter selection independent of the unknown noise level. We established weak decomposability of the clustered Lasso penalty and a noise-level-free tuning strategy under Gaussian noise. We further established finite-sample prediction error guarantees for the resulting estimator under suitable regularity conditions.
	We noted that the objective function was nonsmooth, which introduced substantial computational difficulties. To address this issue, we developed an efficient semismooth Newton-based proximal augmented Lagrangian algorithm (N-ALM) for the dual problem and established global convergence and a local linear rate for the distance to the KKT solution set under the stated conditions, with a superlinear rate when the penalty parameter tends to infinity. Numerical results on real datasets showed that N-ALM efficiently computed accurate solutions of large-scale SRCL problems and substantially reduced running times compared with d-ADMM.
	For future work, we think that it should be worthwhile to explore more statistical properties of the proposed estimator, including support recovery, oracle inequalities and asymptotic distributions. 
	
	\section*{Declarations}
	
	\textbf{Ethics approval and consent to participate.} Not applicable.
	
	\noindent\textbf{Consent for publication.} Not applicable.
	
	\noindent\textbf{Funding.} The work of Y. Xiao is supported by the National Natural Science Foundation of China (Grant No. 12471307 and 12271217).
	
	\noindent\textbf{Availability of data and materials.} Only public datasets, in \href{https://www.csie.ntu.edu.tw/~cjlin/libsvmtools/datasets/}{www.csie.ntu.edu.tw/cjlin/libsvmtools/datasets/}, are used.
	
	\noindent\textbf{Competing interests.} The authors declare that they have no conflict of interest.
	
	\noindent\textbf{Authors’ contributions.} S. Wang conducted the literature review and model derivation, and developed the algorithm framework as the first author. M. Shi carried out the theoretical analysis. Y. Xiao conceptualized the research, supervised the study, acquired funding, drafted the main manuscript, and acted as the corresponding author.	All authors reviewed and approved the final manuscript.

	\bibliographystyle{spbasic}
	\bibliography{hcf2}
	
\end{document}